\documentclass[11pt]{article}
\usepackage[a4paper,margin=1in]{geometry}
\usepackage[T1]{fontenc}
\usepackage[utf8]{inputenc}
\usepackage[english]{babel}
\usepackage{amsmath,amssymb,amsthm,mathtools}
\usepackage[colorlinks]{hyperref}

\newtheorem{thm}{Theorem}[section]
\newtheorem{lem}[thm]{Lemma}
\newtheorem{prop}[thm]{Proposition}
\newtheorem{cor}[thm]{Corollary}
\newtheorem{defn}[thm]{Definition}
\newtheorem{remark}[thm]{Remark}

\DeclareMathOperator{\supp}{supp}

\title{An optimal trace estimate for microlocal square functions on quadratic surfaces}
\author{Vicente Vergara\footnote{Department of Mathematics, Faculty of Physical and Mathematical Sciences, University of Concepci\'on, Concepci\'on, Chile. \texttt{vvergaraa@udec.cl}}}
\date{}

\begin{document}
\maketitle

\begin{abstract}
We study a local trace estimate for the microlocal angular square function
\[
\mathcal G_R f
:=
\left(\sum_\Theta |f_\Theta|^2\right)^{1/2}
\]
associated with a parabolic decomposition of the frequency annulus of radius $R$ in $\mathbb{R}^3$. The measure under consideration is
\[
\mu_Q=\chi\,\mathcal H^2\lfloor S_Q,
\]
where $\chi\in L^\infty(S_Q)$ is a measurable nonnegative density compactly supported in the patch, and
\[
S_Q=\{(u_1,u_2,Q(u_1,u_2)):u\in U\},
\qquad
Q(u_1,u_2)=\frac12(\lambda_1u_1^2+\lambda_2u_2^2),
\qquad
\lambda_1\lambda_2 >0.
\]
Writing $\rho=R^{-1/2}$, we prove
\[
\|\mathcal G_R f\|_{L^2(\mathrm d\mu_Q)}
\lesssim
R^{1/8}\|f\|_{L^2(\mathbb R^3)}.
\]
Under local positivity of the density near the tangency point, the factor
$R^{1/8}$ is attained by a tangent wave packet test and hence cannot be
improved within this elliptic quadratic model, at this parabolic scale and
for this angular square function. In particular, it measures the failure of a
trace bound uniform in $R$ within this class. Its source is the extreme
tangential interaction between a tube of radius $\rho$ and $S_Q$: the relevant
surface measure is $\sim\rho^{3/2}$, whereas an $L^2$-normalized wave packet
has quadratic size $\sim\rho^{-2}$. Thus the optimal quadratic cost is
$\rho^{-1/2}$, producing the norm factor $\rho^{-1/4}=R^{1/8}$.
\end{abstract}



\section{Introduction}\label{sec:introduccion}

This paper studies, in a nondegenerate quadratic model, the optimal scale of a local trace estimate for a microlocal angular square function over a positive surface measure in $\mathbb{R}^{3}$. The frequency scale is given by a large parameter $R\ge 1$. We consider functions $f$ whose Fourier transform is supported in the spherical annulus of thickness comparable to $1$,
\[
\supp \widehat f
\subset
\Sigma_R
:=
\{\xi\in\mathbb{R}^{3}: ||\xi|-R|\le 1\}.
\]
The frequency geometry is fixed at the parabolic scale: $\Sigma_R$ is covered by frequency boxes $\Theta$ of dimensions comparable to
\[
R^{1/2}\times R^{1/2}\times 1,
\]
with angular aperture comparable to $R^{-1/2}$ and finite overlap.

Let $\{\psi_\Theta\}_\Theta$ be a smooth partition of unity subordinate to this covering. For $f$ with $\supp\widehat f\subset \Sigma_R$, define
\[
\widehat{f_\Theta}(\xi)
=
\psi_\Theta(\xi)\widehat f(\xi),
\qquad
f=\sum_\Theta f_\Theta,
\]
and the associated angular square function by
\[
\mathcal G_R f(x)
:=
\left(
\sum_\Theta |f_\Theta(x)|^2
\right)^{1/2}.
\]

Square functions of this type are related to Stein square functions,
Bochner--Riesz operators, restriction theory, and local smoothing estimates. In those problems, the typical objective is to obtain global bounds in $L^p(\mathbb{R}^n)$ spaces for operators associated with curved surfaces, radial multipliers, or oscillatory integrals; see, for instance, \cite{Stein1993HarmonicAnalysis,Sogge1993FourierIntegrals,GanOhWu2025SteinSquareFunctions}. After uncertainty-principle considerations, such questions often lead to the geometry of families of rectangles, tubes, or Kakeya-type configurations; see \cite{Fefferman1971BallMultiplier,Cordoba1982GeometricFourierAnalysis}. Microlocal decompositions by wave packets and the corresponding almost orthogonality are part of the standard language of this theory; see also \cite{Hormander1985AnalysisIII,BourgainDemeter2015Decoupling}.

The perspective of this paper is closer to geometric obstruction examples than to a global positive theory. In Fefferman's classical counterexample for the ball multiplier, the obstruction is tied to Kakeya-type configurations, where many directions interact in an essential way \cite{Fefferman1971BallMultiplier}. In restriction problems for eigenfunctions to submanifolds, analogous phenomena appear in the form of microlocal concentration in tubes or Gaussian beams, and the resulting estimates may be optimal in the presence of suitable extremal examples \cite{BurqGerardTzvetkov2007Restrictions,Hu2009Restriction,BlairSogge2016KakeyaNikodym}.

The phenomenon isolated here is more elementary. No Kakeya-type configuration with many orientations is required. In the quadratic model considered below, a single tangent direction and one wave packet adapted to the corresponding tube already produce polynomial growth in the scale. Thus curvature by itself does not eliminate tangent tubular concentration when the output norm is taken with respect to a surface measure in physical space.

The estimate studied in this paper is not a global $L^p$ estimate for Stein's square function, but a local microlocal trace estimate for the angular square function over a singular geometric measure in physical space. The relevant issue is not frequency almost orthogonality in $L^2(\mathbb{R}^{3})$, but the amount of surface measure captured by the physical wave packets.

From this point of view, the problem is close in spirit to restriction inequalities with geometric measures, where the relevant norm is taken with respect to a singular measure; see, for example, \cite{Mockenhaupt2000SalemSets,Mattila2015FourierHausdorff}. However, the object studied here is not a Fourier restriction estimate to a fractal measure in frequency. The measure $\mu$ appears in physical space, and the operator is the angular square function associated with a parabolic decomposition of the frequency annulus.

With the notation fixed in Section~\ref{sec:preliminares}, the relevant wave packets are adapted to tubes of length comparable to $1$ and transverse radius $\rho$. The obstruction comes from the amount of surface measure captured by such tubes in tangential position.

\begin{remark}[comparison with trace estimates]\label{rem:semiclassical-trace-comparison}
The scale $R^{1/8}$ is distinct from the loss in semiclassical trace estimates or eigenfunction restriction estimates to submanifolds; see, for instance, \cite{Sogge1993FourierIntegrals,BurqGerardTzvetkov2007Restrictions}. In those estimates one controls the restriction of a function localized at frequency $h^{-1}\sim R$ to a fixed submanifold. Here, by contrast, the estimate is diagonal with respect to a parabolic angular decomposition of the annulus of thickness $1$.

At the parabolic scale, the relevant packets are concentrated in tubes of length comparable to $1$ and transverse radius $\rho$. The loss comes from an extreme tangential interaction with the surface measure, rather than from the usual semiclassical trace mechanism.
\end{remark}

A reference case is the transversal flat model. If
\[
\mu_\Pi=\chi\,\mathcal H^2\lfloor \Pi,
\]
where $\Pi\subset\mathbb R^3$ is a fixed plane and $\chi\in L^\infty(\Pi)$ is nonnegative and compactly supported, then the projective mechanism gives a uniform bound only for those boxes whose physical directions $v_\Theta$ are uniformly transverse to $\Pi$. Indeed, under the hypothesis
\[
|n_\Pi\cdot v_\Theta|\ge c_0>0,
\]
the orthogonal projections $\pi_{v_\Theta}:\Pi\to v_\Theta^\perp$ have Jacobian uniformly bounded above and below. Therefore the pushforward measure $(\pi_{v_\Theta})_\#\mu_\Pi$ satisfies
\[
(\pi_{v_\Theta})_\#\mu_\Pi(B_{v_\Theta^\perp}(z,r))
\lesssim r^2
\]
uniformly in $\Theta$, $z$, and $0<r\le1$. The mechanism of Proposition~\ref{prop:diagonal-under-projection} then applies with $\beta=2$ and gives a uniform quadratic cost.

Without this transversality, the uniform statement is false. If a tube of length comparable to $1$ and radius $\rho$ has its axis contained in $\Pi$, an $L^2(\mathbb R^3)$-normalized wave packet may have quadratic size $\sim\rho^{-2}$ in the tube, whereas the planar mass of the intersection is $\sim\rho$. This gives a quadratic cost $\rho^{-1}$, which is not uniform. Thus the flat model provides a transversal reference case, rather than a
global bound for all directions.

The main model of this paper is a nondegenerate quadratic patch
\[
S_Q
=
\{(u_1,u_2,Q(u_1,u_2)):u=(u_1,u_2)\in U\},
\]
where
\[
Q(u_1,u_2)
=
\frac12(\lambda_1u_1^2+\lambda_2u_2^2),
\qquad
\lambda_1\lambda_2 >0,
\]
and $U\subset\mathbb{R}^{2}$ is a bounded patch. The measure under consideration is
\[
\mu_Q
=
\chi\,\mathcal H^2\lfloor S_Q,
\]
where $\chi\in L^\infty(S_Q)$ is a measurable nonnegative density compactly supported in the patch. For the optimality argument, it is additionally assumed that $\chi$ is bounded below by a positive constant in a surface neighbourhood of the tangency point under consideration.

\begin{thm}[main result for the elliptic quadratic model]\label{thm:intro-quadratic-main-result}
Let
\[
S_Q=
\{(u_1,u_2,Q(u_1,u_2)):u\in U\},
\qquad
Q(u_1,u_2)
=
\frac12(\lambda_1u_1^2+\lambda_2u_2^2),
\qquad
\lambda_1\lambda_2>0,
\]
and let
\[
\mu_Q=\chi\,\mathcal H^2\lfloor S_Q,
\]
where $\chi\in L^\infty(S_Q)$ is measurable, nonnegative, and compactly supported in the patch. There exists $R_0\ge1$, depending only on the model and on the constants of the partition, such that for every $R\ge R_0$ and every function $f$ with
\[
\supp \widehat f\subset \Sigma_R
\]
one has
\[
\|\mathcal G_R f\|_{L^2(\mathrm d\mu_Q)}
\le
C R^{1/8}
\|f\|_{L^2(\mathbb R^3)}.
\]
The constant $C$ is independent of $R$ and $f$.
For $1\le R<R_0$, the same estimate is absorbed after modifying $C$.

Moreover, if $\chi$ is bounded below by a positive constant in a surface neighbourhood of a tangency point, then there exists a family of functions $f_R$, microlocalized in a single frequency box of the parabolic decomposition, such that
\[
\|\mathcal G_R f_R\|_{L^2(\mathrm d\mu_Q)}
\ge
c R^{1/8}
\|f_R\|_{L^2(\mathbb R^3)}.
\]
Consequently, the exponent $1/8$ is optimal within this elliptic quadratic model, for this parabolic scale and for this angular square function.
\end{thm}

The upper bound in Theorem~\ref{thm:intro-quadratic-main-result} is proved
in the diagonal form of Theorem~\ref{thm:quadratic-diagonal-sharp}. The lower
bound is obtained from a single tangent wave packet constructed in
Section~\ref{sec:bad-geometry}, together with persistence under the angular
partition from Subsection~\ref{subsec:wave-packets-y-tubos}.

The proof separates two mechanisms. In the transversal regime, projection
onto the plane perpendicular to the tube direction gives uniform quadratic
cost. In the tangential regime, the extreme strip has mass $\rho^{3/2}$
inside a tangent tube; combined with the packet normalization
$|\phi_T|^2\sim\rho^{-2}$, this gives the residual quadratic cost
$\rho^{-1/2}$ and hence the norm loss $\rho^{-1/4}=R^{1/8}$.

\subsection{Organization of the manuscript}\label{subsec:esquema-del-trabajo}

In Section~\ref{sec:preliminares} we fix the notation, the frequency localization, the angular square function, the quadratic model $S_Q$, the measure $\mu_Q$, and the physical scale $\rho=R^{-1/2}$.

In Section~\ref{sec:abstract-diagonal-reduction} we record the abstract mechanism that converts projective information, tubular overlap, and almost orthogonality of wave packets into diagonal estimates. The abstract calculation produces the quadratic cost $\rho^{-1/2}$ when the effective projective exponent drops from $2$ to $3/2$.

In Section~\ref{sec:resultados-esperados} we formulate the main results and the estimation mechanisms. There we separate the tubular/projective mechanism from the optimal scale of the quadratic model.

In Section~\ref{sec:bad-geometry} we prove the local geometric estimate responsible for tangential concentration:
\[
\mu_Q(S_Q\cap T_\rho\cap E_v^{\mathrm{ext}})
\sim
\rho^{3/2}.
\]
This is the component that identifies the residual cost.

In Section~\ref{sec:reduccion-wave-packets} we carry out the analytic assembly. The uniform transversal bound and the nontransversal bound of cost $\rho^{-1/2}$ imply the estimate
\[
\|\mathcal G_R f\|_{L^2(\mathrm{d}\mu_Q)}
\lesssim
\rho^{-1/4}
\|f\|_{L^2(\mathbb{R}^{3})}
=
R^{1/8}
\|f\|_{L^2(\mathbb{R}^{3})},
\]
and the tangent wave packet test proves the optimality of this scale.
\section{Preliminaries}\label{sec:preliminares}

\subsection{Basic notation}\label{subsec:notacion-basica}

We write $A\lesssim B$ if there exists a constant $C>0$, independent of the relevant parameters in the context, such that $A\leq C B$. If the constant depends on a parameter $\varepsilon$, we write $A\lesssim_\varepsilon B$. The notation $A\sim B$ means that $A\lesssim B$ and $B\lesssim A$.

For $x_0\in\mathbb{R}^3$ and $r>0$, $B(x_0,r)$ denotes the Euclidean ball centered at $x_0$ with radius $r$. For a set $E\subset \mathbb{R}^3$ and $\rho>0$, we write
\[
N_\rho(E):=\{x\in\mathbb{R}^3:\operatorname{dist}(x,E)\leq \rho\}.
\]

If $\mu$ is a positive Borel measure on a metric space and $E$ is a Borel set, we write
\[
\mu\lfloor E
\]
for the restriction of $\mu$ to $E$, that is,
\[
(\mu\lfloor E)(A):=\mu(A\cap E)
\]
for every Borel set $A$.

If $F:X\to Y$ is a Borel map between metric spaces and $\mu$ is a positive Borel measure on $X$, we denote by
\[
F_\#\mu
\]
the pushforward measure on $Y$, defined by
\[
(F_\#\mu)(B):=\mu(F^{-1}(B))
\]
for every Borel set $B\subset Y$.

The Fourier transform is normalized by
\[
\widehat f(\xi)
=
\int_{\mathbb{R}^3} f(x)e^{-2\pi i x\cdot \xi}\,\mathrm{d}x.
\]

\subsection{Frequency localization, caps, and frequency boxes}\label{subsec:localizacion-frecuencial}

Let $R\geq 1$ be large. We regard $R$ as the frequency scale and use the parabolic convention
\[
\rho:=R^{-1/2}.
\]
Thus $\rho$ is the corresponding transverse physical scale. We consider the spherical annulus
\[
\Sigma_R:=N_1(\{\xi\in\mathbb{R}^3: |\xi|=R\})
=
\{\xi\in\mathbb{R}^3: ||\xi|-R|\leq 1\}.
\]
Frequency boxes of dimensions comparable to
\[
R^{1/2}\times R^{1/2}\times 1
=
\rho^{-1}\times\rho^{-1}\times 1
\]
are dual, by the uncertainty principle, to physical tubes of length comparable to $1$ and transverse radius comparable to $\rho$. This convention remains in force throughout the paper.

Fix a finite family of angular caps $\vartheta$ on the sphere $\{\xi\in\mathbb{R}^3: |\xi|=R\}$, with angular aperture comparable to $R^{-1/2}$ and finite overlap. To each cap $\vartheta$ we associate a frequency box $\Theta$, contained in $\Sigma_R$, of approximate dimensions
\[
R^{1/2}\times R^{1/2}\times 1.
\]
Equivalently, and more precisely, $\Theta$ is an anisotropic box adapted to the local geometry of the annulus, containing the radial thickening of the corresponding cap. The total number of frequency boxes is of order $R$.

We choose a smooth partition of unity subordinate to this covering with an interior plateau, using a standard smooth cutoff construction on finitely overlapping anisotropic boxes; see, for instance, \cite{Stein1993HarmonicAnalysis,Hormander1985AnalysisIII}. More precisely, after fixing a constant dilation of the boxes, we take functions $\psi_\Theta\in C^\infty_c(C\Theta)$ such that
\[
0\le \psi_\Theta\le 1,
\qquad
\sum_\Theta \psi_\Theta(\xi)=1
\qquad
\text{on }\Sigma_R,
\]
with uniformly finite overlap, and so that for each $\Theta$ there exists a concentric subbox
\[
\Theta^\circ\Subset \Theta,
\]
separated from the boundary of $\Theta$ by a fixed fraction of its anisotropic dimensions, such that
\[
\psi_\Theta(\xi)=1
\qquad
\text{for every }\xi\in\Theta^\circ.
\]
All constants associated with this choice, including the finite overlap, the rescaled derivative bounds, and the separation fraction of $\Theta^\circ$ from $\partial\Theta$, are fixed once and for all. For $f$ with $\supp\widehat f\subset\Sigma_R$, set
\[
\widehat{f_\Theta}(\xi):=\psi_\Theta(\xi)\widehat f(\xi).
\]
Then
\[
f=\sum_\Theta f_\Theta.
\]

At this scale the curvature of the sphere is visible at the resolution compatible with the thickness of the annulus.

\subsection{Quadratic model and surface measure}\label{subsec:modelo-cuadratico-medida}

The model under consideration is a nondegenerate elliptic quadratic patch with positive surface density.

Let $U\subset \mathbb{R}^{2}$ be a bounded domain and let
\[
Q(u_1,u_2)
=
\frac12(\lambda_1u_1^2+\lambda_2u_2^2),
\qquad
\lambda_1\lambda_2>0.
\]
Define
\[
S_Q
:=
\{(u_1,u_2,Q(u_1,u_2)):u=(u_1,u_2)\in U\}.
\]
The measure is
\[
\mu_Q
=
\chi\,\mathcal{H}^{2}\lfloor S_Q,
\]
where $\chi\in L^\infty(S_Q)$ is a measurable nonnegative function compactly supported in the patch. For the optimality argument, it will additionally be assumed that $\chi$ is bounded below by a positive constant in a surface neighbourhood of the tangency point under consideration.

Since $\operatorname{supp}\chi$ is compact in the patch, we fix once and for all a relatively compact open set $U_0\Subset U$ such that
\[
\operatorname{supp}\chi\subset X(U_0),
\qquad
X(u_1,u_2)=(u_1,u_2,Q(u_1,u_2)).
\]
All geometric constants associated with the model may depend on $Q$, on $U_0$, on $\|\chi\|_{L^\infty(S_Q)}$, and on the constants fixed in the angular partition, including the finite overlap, the rescaled derivative bounds, and the choice of the interior subboxes $\Theta^\circ\Subset\Theta$. They do not depend on $R$, on $\rho$, or on the frequency box under consideration.

These constants are not taken uniformly under degeneration of the elliptic model. In particular, no uniformity is asserted when
\[
\min\{|\lambda_1|,|\lambda_2|\}\downarrow0,
\]
or when the relatively compact open set $U_0$ approaches the boundary of the chart $U$.

The condition $\lambda_1\lambda_2>0$ is part of the model. It excludes real asymptotic directions and, in particular, excludes rectilinear generators contained in $S_Q$. This restriction is structural for the upper bound with loss $R^{1/8}$: in the hyperbolic case, tubes around generators may have surface mass of order larger than the elliptic extreme scale $\rho^{3/2}$.

\subsection{Structural wave packets and physical tubes}\label{subsec:wave-packets-y-tubos}

The preceding frequency decomposition admits a microlocal decomposition into wave packets adapted to physical tubes dual to the boxes $\Theta$. These tubes have length comparable to $1$ and transverse radius comparable to $\rho$.

The following formulation will be used as a microlocal input in the later sections. It is the phase--space decomposition associated with smooth anisotropic partitions of the annulus, with $L^2$ almost orthogonality and rapid decay away from the corresponding physical tube; see, for example, \cite{Stein1993HarmonicAnalysis,Sogge1993FourierIntegrals,Hormander1985AnalysisIII}.

\begin{prop}[structural wave packet decomposition]\label{prop:standard-wave-packet-decomposition}
Let $\Theta$ be a frequency box of dimensions comparable to $\rho^{-1}\times\rho^{-1}\times 1$ inside $\Sigma_R$. Let $v_\Theta\in S^2$ be the central physical direction dual to $\Theta$, and fix an orthonormal basis
\[
(e_{\Theta,1},e_{\Theta,2},v_\Theta),
\qquad
e_{\Theta,1},e_{\Theta,2}\in v_\Theta^\perp.
\]
For $m=(m_1,m_2)\in\mathbb Z^2$ and $j\in\mathbb Z$, define
\[
x_{\Theta,m,j}
:=
\rho m_1 e_{\Theta,1}
+
\rho m_2 e_{\Theta,2}
+
jv_\Theta,
\]
and let $T_{\Theta,m,j}$ be the tube with direction $v_\Theta$, length comparable to $1$, and transverse radius comparable to $\rho$, given by
\[
T_{\Theta,m,j}
:=
\left\{
x=y+s v_\Theta:
|y-\rho m_1e_{\Theta,1}-\rho m_2e_{\Theta,2}|\le C\rho,\ |s-j|\le C
\right\},
\]
where $y\in v_\Theta^\perp$ and $C\ge1$ is a fixed geometric constant.

Then there exists a family of functions
\[
\{\phi_{\Theta,m,j}\}_{m\in\mathbb Z^2,\ j\in\mathbb Z}
\subset \mathcal S(\mathbb R^3)
\]
and a dual family
\[
\{\widetilde\phi_{\Theta,m,j}\}_{m\in\mathbb Z^2,\ j\in\mathbb Z}
\subset \mathcal S(\mathbb R^3)
\]
with the following property. For every function $f_\Theta$ with
\[
\supp\widehat{f_\Theta}\subset \Theta,
\]
one has the expansion
\[
f_\Theta
=
\sum_{m\in\mathbb Z^2}
\sum_{j\in\mathbb Z}
c_{\Theta,m,j}\phi_{\Theta,m,j},
\qquad
c_{\Theta,m,j}:=
\langle f_\Theta,\widetilde\phi_{\Theta,m,j}\rangle,
\]
with convergence in $L^2(\mathbb R^3)$, and
\[
\sum_{m\in\mathbb Z^2}
\sum_{j\in\mathbb Z}
|c_{\Theta,m,j}|^2
\le
C
\|f_\Theta\|_{L^2(\mathbb R^3)}^2,
\]
where $C$ is independent of $R$ and $\Theta$.

Moreover, for every $m\in\mathbb Z^2$, $j\in\mathbb Z$, and every $N\ge1$, the following properties hold:

\begin{enumerate}
\item
\[
\supp\widehat{\phi_{\Theta,m,j}}\subset C\Theta,
\]
for a fixed dilation $C\Theta$.

\item
\[
\|\phi_{\Theta,m,j}\|_{L^2(\mathbb R^3)}\le C.
\]

\item
For every $x\in\mathbb R^3$,
\[
|\phi_{\Theta,m,j}(x)|
\le
C_N\rho^{-1}
\left(
1+
\frac{|\pi_{v_\Theta}x-\rho m_1e_{\Theta,1}-\rho m_2e_{\Theta,2}|}{\rho}
+
|x\cdot v_\Theta-j|
\right)^{-N},
\]
where $\pi_{v_\Theta}$ denotes the orthogonal projection onto $v_\Theta^\perp$.

\item
The same type of estimate holds for $\widetilde\phi_{\Theta,m,j}$, with constants uniform in $R$, $\Theta$, $m$, and $j$.
\end{enumerate}

In particular, if we write $T=T_{\Theta,m,j}$ and $\phi_T=\phi_{\Theta,m,j}$, then
\[
f_\Theta=\sum_{T\in\mathbb T(\Theta)} c_T\phi_T,
\qquad
\sum_{T\in\mathbb T(\Theta)}|c_T|^2
\lesssim
\|f_\Theta\|_{L^2(\mathbb R^3)}^2,
\]
where
\[
\mathbb T(\Theta)
=
\{T_{\Theta,m,j}:m\in\mathbb Z^2,\ j\in\mathbb Z\}.
\]
\end{prop}

\begin{remark}[longitudinal index and compact support]\label{rem:wave-packet-longitudinal-index}
The longitudinal index $j$ is part of the decomposition in $L^2(\mathbb R^3)$. However, the estimates with respect to $\mu_Q$ involve integration over the fixed compact set $\operatorname{supp}\mu_Q\subset X(U_0)$. Therefore, for each direction $v_\Theta$ and each point $x\in\operatorname{supp}\mu_Q$, the effective longitudinal sum satisfies
\[
\sum_{j\in\mathbb Z}(1+|x\cdot v_\Theta-j|)^{-N}\le C_N,
\]
uniformly in $\Theta$ and in $x$. In particular, the relevant overlap counts on $\operatorname{supp}\mu_Q$ are transversal: the growth under transverse dilation of the tubes by a factor $2^k$ is $2^{2k}$, not $2^{3k}$.
\end{remark}

\begin{lem}[model packet in a frequency box]\label{lem:model-wave-packet-in-box}
Let $\Theta_0$ be a frequency box of dimensions comparable to $\rho^{-1}\times\rho^{-1}\times 1$ inside $\Sigma_R$, with central physical direction $v_0\in S^2$. Let $\xi_0\in\Theta_0$ be a point whose distance from the boundary of the box is comparable to the size of $\Theta_0$, after replacing $\Theta_0$, if necessary, by a concentric subbox of fixed dilation.

Then there exists a tube $T$ of length comparable to $1$, transverse radius comparable to $\rho$, and direction parallel to $v_0$, and there exists a function $\phi_T\in\mathcal S(\mathbb R^3)$ with the following properties.

\begin{enumerate}
\item
The Fourier transform of $\phi_T$ is contained in a fixed dilation of $\Theta_0$:
\[
\supp \widehat{\phi_T}\subset C\Theta_0,
\]
where $C\ge1$ is an absolute constant depending only on the choice of the model profiles.

\item
One has the normalization
\[
\|\phi_T\|_{L^2(\mathbb R^3)}\sim 1,
\]
with constants independent of $R$ and $\Theta_0$.

\item
For every $N\ge1$ there exists $C_N\ge1$, independent of $R$ and $\Theta_0$, such that
\[
|\phi_T(x)|
\le
C_N\rho^{-1}
\left(
1+\frac{\operatorname{dist}(x,T)}{\rho}
\right)^{-N}
\qquad
\text{for every }x\in\mathbb R^3.
\]

\item
There exists a subtube $T'\subset T$, of length comparable to $1$ and transverse radius comparable to $\rho$, such that
\[
|\phi_T(x)|\ge c\rho^{-1}
\qquad
\text{for every }x\in T',
\]
where $c>0$ is independent of $R$ and $\Theta_0$.
\end{enumerate}
\end{lem}

\begin{proof}
Fix an orthonormal basis of $\mathbb R^3$ of the form $(e_1,e_2,v_0)$, and identify each point $x\in\mathbb R^3$ with
\[
x=y+s v_0,
\qquad
y\in v_0^\perp,\quad s\in\mathbb R.
\]
Let $a\in\mathcal S(v_0^\perp)$ be such that $\widehat a$ is supported in a sufficiently small ball of $v_0^\perp$ and $|a(y)|\ge c_a>0$ for $|y|\le c_0$. Let $b\in\mathcal S(\mathbb R)$ be such that $\widehat b$ is supported in a sufficiently small interval and $|b(s)|\ge c_b>0$ for $|s|\le c_0$. Define
\[
\phi_T(x)
:=
\rho^{-1}
e^{2\pi i \xi_0\cdot x}
a\left(\frac{y-y_0}{\rho}\right)
b(s-s_0),
\]
where $y_0\in v_0^\perp$ and $s_0\in\mathbb R$ are fixed, and let
\[
T
:=
\{y+s v_0:\ |y-y_0|\le C\rho,\ |s-s_0|\le C\}.
\]

By a change of variables,
\[
\|\phi_T\|_{L^2(\mathbb R^3)}^2
=
\rho^{-2}
\int_{\mathbb R}
\int_{v_0^\perp}
\left|
a\left(\frac{y-y_0}{\rho}\right)
\right|^2
|b(s-s_0)|^2
\,\mathrm d y\,\mathrm d s
=
\|a\|_{L^2(v_0^\perp)}^2\|b\|_{L^2(\mathbb R)}^2.
\]
After normalizing $a$ and $b$, this gives $\|\phi_T\|_{L^2(\mathbb R^3)}\sim1$.

The Fourier transform of $\phi_T$ is contained in
\[
\xi_0+
\left\{
\eta_\perp+\eta_\parallel v_0:
\eta_\perp\in \rho^{-1}\supp\widehat a,\ 
\eta_\parallel\in\supp\widehat b
\right\}.
\]
Since $\supp\widehat a$ and $\supp\widehat b$ were chosen sufficiently small, this set is contained in a fixed dilation of $\Theta_0$. This proves the frequency localization.

The spatial decay follows from the fact that $a$ and $b$ are Schwartz functions. Indeed, for every $N\ge1$,
\[
\left|
a\left(\frac{y-y_0}{\rho}\right)b(s-s_0)
\right|
\le
C_N
\left(
1+\frac{|y-y_0|}{\rho}
\right)^{-N}
(1+|s-s_0|)^{-N}.
\]
Since the tube $T$ has transverse radius comparable to $\rho$ and length comparable to $1$, this bound implies
\[
|\phi_T(x)|
\le
C_N\rho^{-1}
\left(
1+\frac{\operatorname{dist}(x,T)}{\rho}
\right)^{-N},
\]
after modifying $C_N$.

Finally, if
\[
|y-y_0|\le c_0\rho,
\qquad
|s-s_0|\le c_0,
\]
then, by the choice of $a$ and $b$,
\[
|\phi_T(x)|
=
\rho^{-1}
\left|
a\left(\frac{y-y_0}{\rho}\right)
\right|
|b(s-s_0)|
\ge
c_a c_b\rho^{-1}.
\]
The set of such points contains a subtube $T'\subset T$ of length comparable to $1$ and transverse radius comparable to $\rho$. This concludes the proof.
\end{proof}

\begin{lem}[persistence under the angular partition]\label{lem:model-wave-packet-square-function-lower}
Let $\Theta_0$ be a frequency box in the decomposition of $\Sigma_R$. Suppose that $\Theta_0'\Subset\Theta_0$ is a concentric subbox separated from the boundary of $\Theta_0$ by a fixed fraction of the dimensions of $\Theta_0$. Let $\phi_T$ be a model packet as in Lemma~\ref{lem:model-wave-packet-in-box}, constructed with sufficiently small frequency profiles so that
\[
\supp\widehat{\phi_T}\subset \Theta_0'.
\]
Then there exists a constant $c>0$, depending only on the smooth partition $\{\psi_\Theta\}_\Theta$ and on the fixed choice of the subbox, such that
\[
\mathcal G_R\phi_T(x)
\ge
c|\phi_T(x)|
\qquad
\text{for every }x\in\mathbb R^3.
\]
In particular, the same bound holds on any subtube where $\phi_T$ has a lower bound comparable to $\rho^{-1}$.
\end{lem}

\begin{proof}
By the choice of the angular partition with an interior plateau, there exists a concentric subbox $\Theta_0^\circ\Subset\Theta_0$, separated from $\partial\Theta_0$ by a fixed fraction of the anisotropic dimensions of $\Theta_0$, such that
\[
\psi_{\Theta_0}(\xi)=1
\qquad
\text{for every }\xi\in\Theta_0^\circ.
\]
After refining the subbox $\Theta_0'$ if necessary, assume
\[
\Theta_0'\subset \Theta_0^\circ.
\]
Since $\supp\widehat{\phi_T}\subset\Theta_0'$, one has
\[
(\phi_T)_{\Theta_0}
=
\mathcal F^{-1}(\psi_{\Theta_0}\widehat{\phi_T})
=
\phi_T.
\]
Therefore,
\[
\mathcal G_R\phi_T(x)
=
\left(\sum_\Theta |(\phi_T)_\Theta(x)|^2\right)^{1/2}
\ge
|(\phi_T)_{\Theta_0}(x)|
=
|\phi_T(x)|.
\]
The case with a lower bound $c|\phi_T|$ follows from the same argument if $\psi_{\Theta_0}$ is not exactly $1$ but is uniformly separated from $0$ on the frequency support of the packet, with the fixed multiplier incorporated into the definition of the profile.
\end{proof}

\begin{remark}[central direction and parallelization]\label{rem:central-direction-parallelization}
In the later sections, tubes in $\mathbb T(\Theta)$ will often be treated as tubes parallel to the central direction $v_\Theta$. This convention does not introduce an additional geometric hypothesis.

Indeed, by Proposition~\ref{prop:standard-wave-packet-decomposition}, if $T\in\mathbb T(\Theta)$ has actual direction $v_T$, then $|v_T-v_\Theta|\lesssim\rho$. Since the tubes under consideration have length comparable to $1$, the angular deviation $O(\rho)$ produces a transverse displacement $O(\rho)$ along the whole tube. Hence each tube $T$ is contained in a constant dilation of a tube $\widetilde T$ of radius comparable to $\rho$ and direction exactly parallel to $v_\Theta$.

The dyadic overlap estimates, projective bounds, and Schur arguments used below are stable under this replacement, at the cost of modifying only the implicit geometric constants. In particular, the use of tubular families parallel to $v_\Theta$ is a normalization at scale $\rho$.
\end{remark}

\begin{lem}[tubular stability under central parallelization]\label{lem:central-direction-tube-stability}
Let $\rho\in(0,1]$. Let $v,w\in S^2$ be such that
\[
|v-w|\le A_0\rho
\]
for some fixed constant $A_0\ge1$. Let $T$ be a tube of length comparable to $1$, transverse radius comparable to $\rho$, and direction $v$. Then there exists a tube $\widetilde T$ of length comparable to $1$, transverse radius comparable to $\rho$, and direction $w$, such that
\[
T\subset C\widetilde T,
\]
where $C\ge1$ depends only on $A_0$ and on the geometric constants in the definition of tube.

Moreover, for every $N\ge1$ there exists $C_N\ge1$, with the same dependencies and additionally depending on $N$, such that
\[
\left(
1+\frac{\operatorname{dist}(x,T)}{\rho}
\right)^{-N}
\le
C_N
\left(
1+\frac{\operatorname{dist}(x,\widetilde T)}{\rho}
\right)^{-N}
\]
for every $x\in\mathbb R^3$.
\end{lem}

\begin{proof}
Let $\ell_T$ be the axis of $T$, and let $x_T$ be a fixed point of $\ell_T$ in the central segment of the tube. Define $\widetilde T$ as the tube with direction $w$, length comparable to $1$, and radius comparable to $\rho$, whose axis passes through $x_T$ and whose longitudinal interval contains the projection of the longitudinal segment of $T$ onto the direction $w$.

If $x\in T$, then there exists $s$ with $|s|\lesssim1$ such that
\[
\operatorname{dist}(x,x_T+s v)\lesssim \rho.
\]
Since $|v-w|\le A_0\rho$, one has
\[
|x_T+s v-(x_T+s w)|
\le
|s|\,|v-w|
\lesssim
\rho.
\]
Therefore,
\[
\operatorname{dist}(x,\ell_{\widetilde T})\lesssim \rho,
\]
and the longitudinal coordinate of $x$ with respect to $w$ remains in an interval of length comparable to $1$. Increasing the geometric constants, one obtains
\[
T\subset C\widetilde T.
\]

For the tail estimate, the preceding inclusion implies
\[
\operatorname{dist}(x,\widetilde T)
\le
\operatorname{dist}(x,T)+C\rho.
\]
Hence
\[
1+\frac{\operatorname{dist}(x,\widetilde T)}{\rho}
\le
C\left(
1+\frac{\operatorname{dist}(x,T)}{\rho}
\right).
\]
Equivalently,
\[
\left(
1+\frac{\operatorname{dist}(x,T)}{\rho}
\right)^{-N}
\le
C_N
\left(
1+\frac{\operatorname{dist}(x,\widetilde T)}{\rho}
\right)^{-N}.
\]
This proves the claim.
\end{proof}

\subsection{Wave packets and the optimal scale}\label{subsec:wave-packets-y-escala-sharp}

The preceding normalization implies that an $L^2$-normalized wave packet essentially supported in a tube of length comparable to $1$ and radius $\rho$ has typical size $\rho^{-1}$ on the tube. In the quadratic model, the extremal tangential interaction occurs when such a tube meets the extreme strip
\[
E_v^{\mathrm{ext}}
=
\{x\in S_Q: |n_Q(x)\cdot v|\le \rho^{1/2}\}.
\]
The corresponding geometric mass estimate is proved in
Section~\ref{sec:bad-geometry}; combined with the packet normalization, it
gives the quadratic scale $\rho^{-1/2}$ and hence the norm scale
$\rho^{-1/4}=R^{1/8}$.

\subsection{Geometric setup and normalization}\label{subsec:setup-geometrico-y-normalizacion}

In the rest of the manuscript we adopt the fixed parabolic normalization
\[
\rho:=R^{-1/2}.
\]
Thus every frequency box $\Theta$ has dimensions comparable to
\[
\rho^{-1}\times \rho^{-1}\times 1
=
R^{1/2}\times R^{1/2}\times 1,
\]
and the dual physical tubes have length comparable to $1$ and transverse radius comparable to $\rho$.

For each frequency box $\Theta$, fix a unit direction $v_{\Theta}\in S^{2}$, an approximate normal to the underlying spherical patch. We denote by $v_{\Theta}^{\perp}$ the plane orthogonal to $v_{\Theta}$, and by
\[
\pi_{v_{\Theta}}:\mathbb{R}^{3}\to v_{\Theta}^{\perp}
\]
the corresponding orthogonal projection.

The family of tubes associated with $\Theta$ is denoted by $\mathbb{T}(\Theta)$. Each $T\in\mathbb{T}(\Theta)$ is a tube of length comparable to $1$ and transverse radius comparable to $\rho$, oriented in the direction $v_{\Theta}$. Whenever an explicit discretization is required, we assume that the tube axes are indexed by a lattice of spacing comparable to $\rho$ in $v_{\Theta}^{\perp}$, with uniformly bounded overlap at unit scale in the longitudinal variable.

The later analytic formulation will be written in terms of representations of the form
\[
f_{\Theta}=\sum_{T\in\mathbb{T}(\Theta)} c_{T}\phi_{T},
\]
where the functions $\phi_{T}$ simultaneously satisfy the following quantitative properties:

\begin{enumerate}
\item
\textbf{Quadratic normalization:}
\[
\|\phi_{T}\|_{L^{2}(\mathbb{R}^{3})}\le 1.
\]

\item
\textbf{Essential frequency localization:} the Fourier transform $\widehat{\phi_{T}}$ decays rapidly away from a constant dilation of $\Theta$.

\item
\textbf{Spatial adaptation to the tube:} there exist $A\ge 1$ and, for each sufficiently large integer $N\ge 1$, a constant $C_{N}\ge 1$, independent of $\Theta$ and $T$, such that
\[
|\phi_{T}(x)|
\le
C_{N}A\,\rho^{-1}
\left(1+\frac{\operatorname{dist}(x,T)}{\rho}\right)^{-N}
\qquad\text{for every }x\in\mathbb{R}^{3}.
\]
\end{enumerate}

Consequently, the abstract block formulated below is conditional with respect to any family $\{\phi_{T}\}_{T\in\mathbb{T}(\Theta)}$ satisfying these properties. In the applications to the quadratic model, this family will be provided by Proposition~\ref{prop:standard-wave-packet-decomposition}, including the Bessel inequality, the anisotropic decay, and the normalization of the packets.
\section{Abstract reduction for the diagonal estimate}\label{sec:abstract-diagonal-reduction}

\subsection{Axiomatic scheme for the diagonal estimate}\label{subsec:axiomatic-diagonal}

This section isolates the abstract mechanism that converts a projective bound for the measure, tubular overlap, and Bessel orthogonality for wave packets into a diagonal estimate with explicit quadratic cost in the transverse scale $\rho$.

In the quadratic model, the transversal regime corresponds to the case $\beta=2$. The extreme tangential contribution is treated later by tubular Schur rather than by a global projective hypothesis with exponent $3/2$.

\begin{defn}[projective hypothesis]\label{def:projective-hypothesis}
Let $\beta\in(0,2]$. We say that a positive Borel measure $\mu$ satisfies the projective hypothesis $(P_{\beta})$ for a family of directions $\{v_{\Theta}\}_{\Theta}$ if there exists a constant $C_{\mathrm{proj}}>0$ such that, for every frequency box $\Theta$, if $\pi_{v_{\Theta}}:\mathbb{R}^{3}\to v_{\Theta}^{\perp}$ denotes the orthogonal projection onto the plane perpendicular to $v_{\Theta}$, then
\[
(\pi_{v_{\Theta}})_{\#}\mu\bigl(B_{v_{\Theta}^{\perp}}(z,r)\bigr)\le C_{\mathrm{proj}}\,r^{\beta}
\]
for every $z\in v_{\Theta}^{\perp}$ and every $0<r\le 1$.
\end{defn}

\begin{defn}[dyadic overlap hypothesis]\label{def:overlap-hypothesis}
Let $a\ge 0$. We say that a family of tubes $\mathbb{T}(\Theta)$ satisfies the dyadic overlap hypothesis $(O_{a})$ if there exist constants $M\ge 1$ and $C_{0}\ge 1$ such that, for every integer $k\ge 0$,
\[
\sum_{T\in\mathbb{T}(\Theta)} \mathbf{1}_{2^{k}C_{0}T}(x)\le M\,2^{ak}
\qquad\text{for every }x\in\mathbb{R}^{3}.
\]
\end{defn}

\begin{remark}\label{rem:axiomatic-diagonal}
The analytic argument starts from the microlocal decomposition of
Proposition~\ref{prop:standard-wave-packet-decomposition},
\[
f_{\Theta}
=
\sum_{T\in\mathbb{T}(\Theta)} c_{T}\phi_{T},
\]
together with the Bessel inequality for the coefficients. Under the projective hypothesis $(P_\beta)$ and the dyadic overlap hypothesis $(O_a)$, the diagonal mechanism of this section produces an estimate of the form
\[
\int_{\mathbb{R}^{3}} \sum_{\Theta}|f_{\Theta}(x)|^{2}\,\mathrm{d}\mu(x)
\lesssim
\rho^{\beta-2}
\sum_{\Theta}\sum_{T\in\mathbb{T}(\Theta)} |c_{T}|^{2},
\]
where $\rho$ is the transverse radius of the tubes dual to the frequency boxes.

At the parabolic scale, $\rho=R^{-1/2}$, and hence the corresponding cost is
\[
\rho^{\beta-2}=R^{(2-\beta)/2}
\]
at the quadratic level, or equivalently
\[
R^{(2-\beta)/4}
\]
after taking square roots in the $L^{2}(\mathrm{d}\mu)$ norm.
\end{remark}

\begin{remark}[use of the standard decomposition]\label{rem:standard-wave-packet-input}
In the applications to the quadratic model, the wave packet representations used in this section are those of Proposition~\ref{prop:standard-wave-packet-decomposition}. In particular, the Bessel inequality, the anisotropic decay, and the normalization of the packets are taken with the structural constants fixed there.

When the abstract mechanism is formulated with tubes parallel to a central direction $v_\Theta$, one uses the normalization of Remark~\ref{rem:central-direction-parallelization}. Therefore, strict parallelism in the tubular hypotheses is a geometric reduction stable at scale $\rho$, not an additional hypothesis on the microlocal construction.
\end{remark}

\begin{remark}[flat transversal model]\label{rem:flat-model-uniform-cost}
In the transversal flat model, the projections
$\pi_{v_\Theta}:\Pi\to v_\Theta^\perp$ have Jacobian uniformly bounded above
and below. Hence $(\pi_{v_\Theta})_\#\mu_\Pi$ satisfies the projective
hypothesis with $\beta=2$, and Proposition~\ref{prop:diagonal-under-projection}
gives a scale-uniform quadratic estimate.
\end{remark}

\begin{remark}\label{rem:standard-overlap-a2}
For a transversal discretization by a lattice of spacing $\rho$ in $v_{\Theta}^{\perp}$, the natural value of the overlap parameter is
\[
a=2.
\]
When the tubes are transversely dilated by a factor $2^{k}$, the number of lattice axes that can contribute at a fixed point grows like the area of a disk of radius $2^{k}\rho$ in the plane $v_{\Theta}^{\perp}$, namely like $2^{2k}$.
\end{remark}

\begin{lem}[dyadic overlap for tubes with a longitudinal index]\label{lem:dyadic-overlap-standard}
Let $\rho\in(0,1]$, let $v\in S^{2}$, and fix an orthonormal basis of the form $(e_1,e_2,v)$. For $m=(m_1,m_2)\in\mathbb Z^2$ and $j\in\mathbb Z$, define
\[
T_{m,j}
:=
\left\{
x=y+s v:
|y-\rho m_1e_1-\rho m_2e_2|\le C_0\rho,\ |s-j|\le C_0
\right\}.
\]
For $k\ge0$, define the transverse dilation
\[
2^kT_{m,j}
:=
\left\{
x=y+s v:
|y-\rho m_1e_1-\rho m_2e_2|\le C_0 2^k\rho,\ |s-j|\le C_0
\right\}.
\]
Then there exists a constant $C\ge1$, depending only on $C_0$, such that
\[
\sum_{m\in\mathbb Z^2}
\sum_{j\in\mathbb Z}
\mathbf 1_{2^kT_{m,j}}(x)
\le
C2^{2k}
\qquad
\text{for every }x\in\mathbb R^3\text{ and every }k\ge0.
\]
In particular, the dyadic overlap hypothesis of Proposition~\ref{prop:diagonal-under-projection} holds with $a=2$ for this family, provided that the tubular dilations used there are interpreted as transverse dilations.
\end{lem}

\begin{proof}
Write
\[
x=y_x+s_xv,
\qquad
y_x\in v^\perp,\quad s_x\in\mathbb R.
\]
If $x\in 2^kT_{m,j}$, then
\[
|s_x-j|\le C_0,
\]
so there are $O_{C_0}(1)$ possible values of $j$. Moreover,
\[
|y_x-\rho m_1e_1-\rho m_2e_2|
\le
C_0 2^k\rho.
\]
The number of points of the lattice $\rho\mathbb Z^2$ contained in a disk of radius $C_0 2^k\rho$ is $O_{C_0}(2^{2k})$. Multiplying by the uniformly bounded number of longitudinal values gives
\[
\sum_{m\in\mathbb Z^2}
\sum_{j\in\mathbb Z}
\mathbf 1_{2^kT_{m,j}}(x)
\le
C2^{2k}.
\]
\end{proof}

\begin{remark}\label{rem:dyadic-overlap-standard}
The exponent $2$ in Lemma~\ref{lem:dyadic-overlap-standard} is the natural transversal growth in ambient dimension $3$. The longitudinal variable does not produce an additional factor because the length of the tubes is not dilated: for each fixed point there are only $O(1)$ longitudinal values $j$ compatible with the condition $|x\cdot v-j|\lesssim1$. If one also dilated longitudinally by $2^k$, the growth would be $2^{3k}$, and the diagonal argument would produce a different loss. In this manuscript all tubular dilations used for dyadic overlap are transverse.
\end{remark}

\begin{prop}[diagonal estimate under projective and dyadic overlap hypotheses]\label{prop:diagonal-under-projection}
Let $\rho\in(0,1]$. For each frequency box $\Theta$, let $\mathbb{T}(\Theta)$ be a finite family of tubes of length comparable to $1$ and radius $\rho$, all parallel to a unit direction $v_{\Theta}\in S^{2}$. Suppose that one has a representation
\[
f_{\Theta}=\sum_{T\in\mathbb{T}(\Theta)} c_{T}\phi_{T},
\]
where each $\phi_{T}\in L^{2}(\mathbb{R}^{3})$ satisfies
\[
\|\phi_{T}\|_{L^{2}(\mathbb{R}^{3})}\le 1
\]
and
\[
|\phi_{T}(x)| \le A\,\rho^{-1}\Bigl(1+\frac{\operatorname{dist}(x,T)}{\rho}\Bigr)^{-N}
\]
for some constants $A\ge 1$ and $N>0$ independent of $\Theta$ and $T$.

Suppose moreover that there exist constants $M\ge 1$, $a\ge 0$, and $C_{0}\ge 1$ such that, for each $\Theta$ and every integer $k\ge 0$,
\[
\sum_{T\in\mathbb{T}(\Theta)} \mathbf{1}_{2^{k}C_{0}T}(x)\le M\,2^{ak}
\qquad\text{for every }x\in\mathbb{R}^{3}.
\]
Let $\mu$ be a positive Borel measure supported in $B(0,1)\subset\mathbb{R}^{3}$. Suppose also that there exist $\beta\in(0,2]$ and a constant $C_{\mathrm{proj}}>0$ such that, for every frequency box $\Theta$, if $\pi_{v_{\Theta}}:\mathbb{R}^{3}\to v_{\Theta}^{\perp}$ denotes the orthogonal projection, then
\[
(\pi_{v_{\Theta}})_{\#}\mu\bigl(B_{v_{\Theta}^{\perp}}(z,r)\bigr)\le C_{\mathrm{proj}}\,r^{\beta}
\]
for every $z\in v_{\Theta}^{\perp}$ and every $0<r\le 1$.

If
\[
N>\frac{\beta+a}{2},
\]
then
\[
\int_{\mathbb{R}^{3}} \sum_{\Theta} |f_{\Theta}(x)|^{2}\,\mathrm{d}\mu(x)
\le
C\,M\,A^{2}C_{\mathrm{proj}}\,\rho^{\beta-2}
\sum_{\Theta}\sum_{T\in\mathbb{T}(\Theta)} |c_{T}|^{2},
\]
where $C\ge 1$ depends only on $N$, on $\beta$, on $a$, and on the geometric constants in the definition of tube.
\end{prop}

\begin{proof}
Fix $\Theta$. For each $T\in\mathbb{T}(\Theta)$ and each integer $k\ge 0$, define
\[
E_{0}(T):=\Bigl\{x\in\mathbb{R}^{3}:\operatorname{dist}(x,T)\le C_{0}\rho\Bigr\},
\]
and, for $k\ge 1$,
\[
E_{k}(T):=\Bigl\{x\in\mathbb{R}^{3}:2^{k-1}C_{0}\rho<\operatorname{dist}(x,T)\le 2^{k}C_{0}\rho\Bigr\}.
\]
Then
\[
\mathbb{R}^{3}=\bigcup_{k\ge 0} E_{k}(T),
\qquad
E_{k}(T)\subset 2^{k}C_{0}T
\]
for every $k\ge 0$.

Let $\sigma>0$ be such that
\[
0<\sigma<2N-a-\beta.
\]
This is possible by the hypothesis $N>(\beta+a)/2$. By the pointwise bound on $\phi_T$, there exists a constant $C\ge 1$ such that
\[
|\phi_{T}(x)|
\le
C\,A\,\rho^{-1}\sum_{k=0}^{\infty} 2^{-Nk}\mathbf{1}_{E_k(T)}(x)
\]
for every $x\in\mathbb{R}^{3}$. Hence
\[
|f_{\Theta}(x)|
\le
C\,A\,\rho^{-1}\sum_{k=0}^{\infty}
2^{-Nk}\sum_{T\in\mathbb{T}(\Theta)} |c_T|\mathbf{1}_{E_k(T)}(x).
\]
Applying Cauchy--Schwarz to the sum in $k$ with weights $2^{-\sigma k}$ gives
\[
|f_{\Theta}(x)|^{2}
\le
C\,A^{2}\rho^{-2}
\sum_{k=0}^{\infty}
2^{(\sigma-2N)k}
\Bigl(\sum_{T\in\mathbb{T}(\Theta)} |c_T|\mathbf{1}_{E_k(T)}(x)\Bigr)^{2}.
\]
For each $k$, another application of Cauchy--Schwarz to the sum in $T$ gives
\[
\Bigl(\sum_{T\in\mathbb{T}(\Theta)} |c_T|\mathbf{1}_{E_k(T)}(x)\Bigr)^{2}
\le
\Bigl(\sum_{T\in\mathbb{T}(\Theta)} |c_T|^{2}\mathbf{1}_{E_k(T)}(x)\Bigr)
\Bigl(\sum_{T\in\mathbb{T}(\Theta)} \mathbf{1}_{E_k(T)}(x)\Bigr).
\]
Since $E_k(T)\subset 2^k C_0T$, the dyadic overlap hypothesis implies
\[
\sum_{T\in\mathbb{T}(\Theta)} \mathbf{1}_{E_k(T)}(x)
\le
\sum_{T\in\mathbb{T}(\Theta)} \mathbf{1}_{2^k C_0T}(x)
\le
M\,2^{ak}.
\]
Substituting this into the previous inequality,
\[
|f_{\Theta}(x)|^{2}
\le
C\,M\,A^{2}\rho^{-2}
\sum_{k=0}^{\infty}
2^{(\sigma+a-2N)k}
\sum_{T\in\mathbb{T}(\Theta)} |c_T|^{2}\mathbf{1}_{E_k(T)}(x).
\]
Integrating with respect to $\mu$ and interchanging the sum and the integral, we obtain
\[
\int_{\mathbb{R}^{3}} |f_{\Theta}(x)|^{2}\,\mathrm{d}\mu(x)
\le
C\,M\,A^{2}\rho^{-2}
\sum_{T\in\mathbb{T}(\Theta)} |c_T|^{2}
\sum_{k=0}^{\infty} 2^{(\sigma+a-2N)k}\mu(E_k(T)).
\]

Let $C_\ast\ge 1$ be a sufficiently large geometric constant, depending only on $C_0$ and on the convention used to dilate tubes. Let $K\ge 0$ be the integer such that
\[
C_\ast 2^K\rho\le 1 < C_\ast 2^{K+1}\rho.
\]
We split the sum in $k$ into two ranges.

For $0\le k\le K$, since $E_k(T)\subset 2^k C_0T$, the projection of $2^k C_0T$ onto $v_\Theta^\perp$ is contained in a ball of radius comparable to $2^k\rho$ centered at the point of $v_\Theta^\perp$ determined by the axis of $T$. By the choice of $K$, for $0\le k\le K$ this radius is at most $1$, and the projective hypothesis gives
\[
\mu(E_k(T))
\le
\mu(2^k C_0T)
\le
C\,C_{\mathrm{proj}}(2^k\rho)^\beta.
\]
Therefore,
\[
\sum_{k=0}^{K} 2^{(\sigma+a-2N)k}\mu(E_k(T))
\le
C\,C_{\mathrm{proj}}\rho^\beta
\sum_{k=0}^{K} 2^{(\sigma+a+\beta-2N)k}.
\]
Since $\sigma+a+\beta-2N<0$, the geometric series converges uniformly in $K$, and we obtain
\[
\sum_{k=0}^{K} 2^{(\sigma+a-2N)k}\mu(E_k(T))
\le
C\,C_{\mathrm{proj}}\rho^\beta.
\]

For $k\ge K+1$, we use only that $E_k(T)\subset \mathbb{R}^{3}$ and that $\mu$ is supported in $B(0,1)$. Indeed, for each $T\in\mathbb{T}(\Theta)$ one has
\[
\mu(E_k(T))\le \mu(B(0,1)).
\]
Thus
\[
\sum_{k=K+1}^{\infty} 2^{(\sigma+a-2N)k}\mu(E_k(T))
\le
\mu(B(0,1))
\sum_{k=K+1}^{\infty} 2^{(\sigma+a-2N)k}.
\]
Since $\sigma+a-2N<0$, the geometric series converges and
\[
\sum_{k=K+1}^{\infty} 2^{(\sigma+a-2N)k}\mu(E_k(T))
\le
C\,\mu(B(0,1))\,2^{(\sigma+a-2N)K}.
\]
Substituting this into the previous estimate for $\int |f_\Theta|^2\,\mathrm{d}\mu$, the large range is bounded by
\[
C\,M\,A^{2}\rho^{-2}\mu(B(0,1))\,2^{(\sigma+a-2N)K}
\sum_{T\in\mathbb{T}(\Theta)} |c_T|^2.
\]
Since $2^K\sim \rho^{-1}$, one has
\[
2^{(\sigma+a-2N)K}\lesssim \rho^{2N-a-\sigma}.
\]
And since $2N-a-\sigma>\beta$, it follows that
\[
\rho^{-2}\,2^{(\sigma+a-2N)K}\lesssim \rho^{\beta-2}.
\]
Consequently, the large range is bounded by
\[
C\,M\,A^{2}\mu(B(0,1))\,\rho^{\beta-2}
\sum_{T\in\mathbb{T}(\Theta)} |c_T|^2.
\]

Combining both ranges, we arrive at
\[
\int_{\mathbb{R}^{3}} |f_{\Theta}(x)|^{2}\,\mathrm{d}\mu(x)
\le
C\,M\,A^{2}\bigl(C_{\mathrm{proj}}+\mu(B(0,1))\bigr)\rho^{\beta-2}
\sum_{T\in\mathbb{T}(\Theta)} |c_{T}|^{2}.
\]
Since $\operatorname{supp}\mu\subset B(0,1)$ and the projective hypothesis holds for every $0<r\le 1$, for any direction $v_{\Theta}$,
\[
\mu(B(0,1))
=
(\pi_{v_{\Theta}})_{\#}\mu(\pi_{v_{\Theta}}(B(0,1)))
\le
(\pi_{v_{\Theta}})_{\#}\mu(B_{v_{\Theta}^{\perp}}(0,1))
\le
C_{\mathrm{proj}}.
\]
Here we used that the orthogonal projection maps $B(0,1)$ into $B_{v_{\Theta}^{\perp}}(0,1)$. Therefore,
\[
\int_{\mathbb{R}^{3}} |f_{\Theta}(x)|^{2}\,\mathrm{d}\mu(x)
\le
C\,M\,A^{2}C_{\mathrm{proj}}\,\rho^{\beta-2}
\sum_{T\in\mathbb{T}(\Theta)} |c_{T}|^{2}.
\]
Finally, summing in $\Theta$, we obtain
\[
\int_{\mathbb{R}^{3}} \sum_{\Theta} |f_{\Theta}(x)|^{2}\,\mathrm{d}\mu(x)
\le
C\,M\,A^{2}C_{\mathrm{proj}}\,\rho^{\beta-2}
\sum_{\Theta}\sum_{T\in\mathbb{T}(\Theta)} |c_{T}|^{2},
\]
as required.
\end{proof}

\begin{cor}[translation to the fixed parabolic scale]\label{cor:diagonal-under-projection}
Under the hypotheses of Proposition~\ref{prop:diagonal-under-projection}, if in addition
\[
\rho=R^{-1/2},
\]
then
\[
\int_{\mathbb{R}^{3}} \sum_{\Theta} |f_{\Theta}(x)|^{2}\,\mathrm{d}\mu(x)
\le
C\,M\,A^{2}C_{\mathrm{proj}}\,R^{(2-\beta)/2}
\sum_{\Theta}\sum_{T\in\mathbb{T}(\Theta)} |c_{T}|^{2}.
\]
In particular, if the wave packet decomposition satisfies
\[
\sum_{\Theta}\sum_{T\in\mathbb{T}(\Theta)} |c_{T}|^{2}\lesssim \sum_{\Theta}\|f_{\Theta}\|_{L^{2}(\mathbb{R}^{3})}^{2},
\]
then
\[
\int_{\mathbb{R}^{3}} \sum_{\Theta} |f_{\Theta}(x)|^{2}\,\mathrm{d}\mu(x)
\le
C\,M\,A^{2}C_{\mathrm{proj}}\,R^{(2-\beta)/2}
\sum_{\Theta}\|f_{\Theta}\|_{L^{2}(\mathbb{R}^{3})}^{2}.
\]
\end{cor}

\begin{proof}
This is the conclusion of Proposition~\ref{prop:diagonal-under-projection} rewritten in the fixed parabolic normalization
\[
\rho=R^{-1/2}.
\]
Indeed,
\[
\rho^{\beta-2}=(R^{-1/2})^{\beta-2}=R^{(2-\beta)/2},
\]
and the first inequality follows by direct substitution. The second follows immediately from the additional hypothesis on the coefficients.
\end{proof}
\section{Main results and estimation mechanisms}\label{sec:resultados-esperados}

This section states the main estimates in the elliptic quadratic model fixed in Subsection~\ref{subsec:modelo-cuadratico-medida}. The local lower bound on $\chi$ near a tangency point is used only for the optimality statement.

\subsection{Main scale of the quadratic model}\label{subsec:quadratic-main-results}

The upper result assembled in Section~\ref{sec:reduccion-wave-packets} has,
in the elliptic model, the form
\[
\|\mathcal G_R f\|_{L^2(\mathrm{d}\mu_Q)}
\lesssim
\rho^{-1/4}\|f\|_{L^2(\mathbb R^3)}
=
R^{1/8}\|f\|_{L^2(\mathbb R^3)}.
\]
In the nontransversal part, the extreme strip is controlled by tubular Schur
using the tail mass estimate, while its complement is controlled by
degenerate projection and projective multiplicity.

This scale is attained by a tangent wave packet. More precisely, there
exists a tangent wave packet $\phi_T$, normalized by
\[
\|\phi_T\|_{L^2(\mathbb R^3)}\sim1,
\qquad
|\phi_T|\gtrsim \rho^{-1}
\quad\text{on }T_\rho,
\]
such that
\[
\|\phi_T\|_{L^2(\mathrm{d}\mu_Q)}
\gtrsim
\rho^{-1/4}\|\phi_T\|_{L^2(\mathbb R^3)}
=
R^{1/8}\|\phi_T\|_{L^2(\mathbb R^3)}.
\]
Thus the preceding upper bound is optimal in the elliptic quadratic model.

\begin{remark}\label{rem:quadratic-results-reading}
The projective estimates, transversal controls, and quadratic assembly identify the scale of the elliptic model. The loss $R^{1/8}$ comes from extreme tangential concentration: the geometric estimate is proved in Section~\ref{sec:bad-geometry}, and the analytic assembly is carried out in Section~\ref{sec:reduccion-wave-packets}.

The hyperbolic case ($\lambda_1\lambda_2<0$) is not part of the optimal result formulated here. The presence of real asymptotic directions may produce straight-line generators contained in the quadratic patch and therefore a different tubular regime. That case requires a separate formulation.
\end{remark}

Let $\Sigma\subset\mathbb R^3$ be a two-dimensional Lipschitz surface and let
$F:\Sigma\to\mathbb R^2$ be a Lipschitz map. At points where $\Sigma$ has an
approximate tangent plane $T_x\Sigma$ and $F$ is approximately differentiable
along $\Sigma$, the tangential Jacobian of $F$ on $\Sigma$ is defined by
\[
J_\Sigma F(x)
:=
\left(
\det\bigl((\mathrm d^\Sigma F(x))(\mathrm d^\Sigma F(x))^{*}\bigr)
\right)^{1/2},
\]
where
\[
\mathrm d^\Sigma F(x):T_x\Sigma\to\mathbb R^2
\]
is the tangential differential and $(\mathrm d^\Sigma F(x))^{*}$ denotes its
Euclidean adjoint. Equivalently, if $(e_1,e_2)$ is an orthonormal basis of
$T_x\Sigma$, then
\[
J_\Sigma F(x)
=
\left|
\partial_{e_1}^{\Sigma}F(x)
\wedge
\partial_{e_2}^{\Sigma}F(x)
\right|.
\]
For $y\in\mathbb R^2$, we denote by
\[
N(y,F,\Sigma)
:=
\#\{x\in\Sigma:F(x)=y\}
\]
the multiplicity function, with values in $\mathbb N\cup\{\infty\}$. In particular, if $\Sigma$ is oriented by a unit normal $n_\Sigma$ and $F=\pi_v|_\Sigma$, where $\pi_v:\mathbb R^3\to v^\perp$ is the orthogonal projection, then
\[
J_\Sigma \pi_v(x)=|n_\Sigma(x)\cdot v|
\]
at every point where $T_x\Sigma$ exists.

\begin{lem}[projective area formula bound]\label{lem:projective-area-formula-bound}

Let $\Sigma\subset\mathbb R^3$ be the graph of a Lipschitz function over a bounded planar domain, and let
\[
\mu=h\,\mathcal H^2\lfloor \Sigma,
\qquad
0\le h\le M.
\]
Let $v\in S^2$ and let $\pi_v:\mathbb R^3\to v^\perp$ be the orthogonal projection. Assume that there are constants $c_0>0$ and $N_0\ge1$ such that
\[
J_{\Sigma}\pi_v(x)\ge c_0
\qquad
\text{for $\mathcal H^2$-a.e. }x\in\Sigma
\]
and
\[
N(y,\pi_v,\Sigma)\le N_0
\qquad
\text{for $\mathcal H^2$-a.e. }y\in v^\perp.
\]
Then there exists a constant $C\ge1$, depending only on $c_0$, $M$, $N_0$, and the geometric constants of the graph, such that
\[
(\pi_v)_\#\mu(B_{v^\perp}(z,r))
\le
C r^2
\]
for every $z\in v^\perp$ and every $r>0$.
\end{lem}

\begin{proof}
Let $B\subset v^\perp$ be a ball. Since $0\le h\le M$,
\[
(\pi_v)_\#\mu(B)
=
\mu(\pi_v^{-1}(B))
\le
M\,\mathcal H^2(\Sigma\cap \pi_v^{-1}(B)).
\]
On $\Sigma$ one has
\[
1\le c_0^{-1}J_{\Sigma}\pi_v(x)
\]
for $\mathcal H^2$-a.e. $x$. Hence
\[
\mathcal H^2(\Sigma\cap \pi_v^{-1}(B))
\le
c_0^{-1}
\int_{\Sigma\cap \pi_v^{-1}(B)}
J_{\Sigma}\pi_v(x)\,\mathrm d\mathcal H^2(x).
\]
By the area formula for Lipschitz mappings, applied to $\pi_v|_\Sigma$,
\[
\int_{\Sigma\cap \pi_v^{-1}(B)}
J_{\Sigma}\pi_v(x)\,\mathrm d\mathcal H^2(x)
=
\int_B
N(y,\pi_v,\Sigma)\,\mathrm d\mathcal H^2_{v^\perp}(y).
\]
See, for instance, \cite[Chapter 3]{EvansGariepy2015MeasureTheory}. Using the multiplicity bound,
\[
\int_B
N(y,\pi_v,\Sigma)\,\mathrm d\mathcal H^2_{v^\perp}(y)
\le
N_0\mathcal H^2_{v^\perp}(B)
\lesssim
N_0 r^2.
\]
Combining the preceding estimates gives the claim.
\end{proof}

\subsection{Tubular/projective mechanism}\label{subsec:natural-projective-route}

The projective mechanism controls the energy of a wave packet by the mass of the orthogonal projection of the measure onto the plane perpendicular to the axis of the tube. If $T$ is a tube of radius $\rho$ and direction $v$, then $\pi_v(T)$ is contained in a ball of radius comparable to $\rho$ in $v^\perp$. Thus the growth of $(\pi_v)_\#\mu$ directly controls the tubular energy.

\begin{lem}\label{lem:tube-projection}
Let $\rho\in(0,1]$, let $v\in S^2$, and let $\pi_v:\mathbb R^3\to v^\perp$ be the orthogonal projection onto the plane perpendicular to $v$. Let $T\subset\mathbb R^3$ be a tube of length comparable to $1$ and radius $\rho$, whose axis is parallel to $v$. Then there exists $y_T\in v^\perp$ such that
\[
T\subset \pi_v^{-1}\bigl(B_{v^\perp}(y_T,C\rho)\bigr),
\]
where $C\ge1$ depends only on the geometric convention used for tubes. Consequently, for every positive Borel measure $\mu$ on $\mathbb R^3$,
\[
\mu(T)\le (\pi_v)_\#\mu\bigl(B_{v^\perp}(y_T,C\rho)\bigr).
\]
\end{lem}

\begin{proof}
Let $\ell_T$ be the axis of $T$. Since $\ell_T$ is a line parallel to $v$, there exists a unique point $y_T\in v^\perp$ such that
\[
\ell_T=\{y_T+tv:t\in\mathbb R\}.
\]
If $x\in T$, then
\[
|\pi_v(x)-y_T|=\operatorname{dist}(x,\ell_T)\le C\rho.
\]
Therefore,
\[
\pi_v(T)\subset B_{v^\perp}(y_T,C\rho),
\]
and the claim follows from the definition of pushforward.
\end{proof}

\begin{lem}[reduction of tubular energy to a projection]\label{lem:wave-packet-projection}
Let $\rho\in(0,1]$, let $v\in S^2$, and let $\pi_v:\mathbb R^3\to v^\perp$ be the orthogonal projection. Let $T\subset\mathbb R^3$ be a tube of length comparable to $1$ and radius $\rho$, whose axis is parallel to $v$. Let $\phi_T\in L^2(\mathbb R^3)$ be such that
\[
|\phi_T(x)|\le A\rho^{-1}
\left(1+\frac{\operatorname{dist}(x,T)}{\rho}\right)^{-N}
\]
for some constants $A\ge1$ and $N>2$. Then
\[
\int_{\mathbb R^3}|\phi_T(x)|^2\,\mathrm{d}\mu(x)
\le
C A^2\rho^{-2}
\sum_{k=0}^{\infty}2^{-2Nk}
(\pi_v)_\#\mu\bigl(B_{v^\perp}(y_T,C2^k\rho)\bigr),
\]
where $y_T\in v^\perp$ is the point associated with the axis of $T$ in Lemma~\ref{lem:tube-projection}, and $C$ depends only on $N$ and on the geometric convention for tubes.
\end{lem}

\begin{proof}
By Lemma~\ref{lem:tube-projection}, $\pi_v(T)\subset B_{v^\perp}(y_T,C_0\rho)$. Moreover, there exists an absolute constant $c>0$ such that
\[
\operatorname{dist}(x,T)\ge c\,\operatorname{dist}(\pi_v(x),\pi_v(T)).
\]
Hence
\[
|\phi_T(x)|^2
\le
C A^2\rho^{-2}
\left(1+\frac{\operatorname{dist}(\pi_v(x),\pi_v(T))}{\rho}\right)^{-2N}.
\]
Integrating with respect to $\mu$ and using the definition of pushforward, the decomposition of $v^\perp$ into annuli centered at $y_T$ at scales $2^k\rho$ gives the stated sum.
\end{proof}

\begin{cor}[projective bound for one wave packet]\label{cor:wave-packet-projection-frostman}
Under the hypotheses of Lemma~\ref{lem:wave-packet-projection}, suppose that there exist $\beta\in(0,2]$ and $C_{\mathrm{proj}}>0$ such that
\[
(\pi_v)_\#\mu\bigl(B_{v^\perp}(z,r)\bigr)
\le
C_{\mathrm{proj}}r^\beta
\]
for every $z\in v^\perp$ and every $0<r\le1$. Suppose also that $\mu$ is supported on a set of diameter bounded by a fixed geometric constant, or else that the estimate is applied locally after restricting $\mu$ to a fixed compact patch. Then, if $N>\beta/2$,
\[
\int_{\mathbb R^3}|\phi_T(x)|^2\,\mathrm{d}\mu(x)
\le
C A^2 C_{\mathrm{proj}}\rho^{\beta-2},
\]
where $C$ depends only on $N$, on $\beta$, on the geometric convention for the tubes, and on the local support constant under consideration.
\end{cor}

\begin{proof}
Apply Lemma~\ref{lem:wave-packet-projection}. In the effective range $2^k\rho\le1$, the projective hypothesis gives
\[
(\pi_v)_\#\mu\bigl(B_{v^\perp}(y_T,C2^k\rho)\bigr)
\le
C C_{\mathrm{proj}}(2^k\rho)^\beta.
\]
Therefore,
\[
\int_{\mathbb R^3}|\phi_T(x)|^2\,\mathrm{d}\mu(x)
\le
C A^2 C_{\mathrm{proj}}\rho^{\beta-2}
\sum_{k\ge0}2^{(\beta-2N)k}.
\]
The series converges because $2N>\beta$. The terms with $2^k\rho>1$ are absorbed into the local constant associated with the patch under consideration.
\end{proof}

\begin{remark}\label{rem:natural-projective-route-limitation}
The projective mechanism produces a uniform quadratic cost when the effective projective exponent is $\beta=2$. In the extreme tangential regime of the quadratic model, the relevant mass is instead
\[
\mu_Q(S_Q\cap T_\rho\cap E_v^{\mathrm{ext}})
\sim \rho^{3/2},
\]
and for a normalized wave packet one has
\[
|\phi_T|^2\sim \rho^{-2}.
\]
The resulting energy has size $\rho^{-1/2}$, corresponding to the norm scale $\rho^{-1/4}=R^{1/8}$.
\end{remark}

\subsection{Transversal projective regime}\label{subsec:transversal-projective-regime}

The transversal regime is the case in which the orthogonal projection to the direction of the tube preserves area on the surface. In that case the effective projective exponent is $\beta=2$, and Corollary~\ref{cor:wave-packet-projection-frostman} produces no scale loss.

\begin{lem}[transversal projection of a surface measure]\label{lem:surface-projection-transverse}
Let $U\subset\mathbb R^2$ be a bounded open set and let $\varphi:U\to\mathbb R$ be a Lipschitz function. Define
\[
\Sigma:=\{(u,\varphi(u)):u\in U\}\subset\mathbb R^3.
\]
Let $\mu=h\,\mathcal H^2\lfloor\Sigma$, where $h:\Sigma\to[0,\infty)$ satisfies
\[
\|h\|_{L^\infty(\Sigma)}\le M.
\]
Let $v\in S^2$ and let $\pi_v:\mathbb R^3\to v^\perp$ be the orthogonal projection. Suppose that there exists a measurable unit normal $n_\Sigma:\Sigma\to S^2$, defined $\mathcal H^2$-almost everywhere, such that
\[
|n_\Sigma(x)\cdot v|\ge c_0>0
\qquad\text{for }\mathcal H^2\text{-almost every }x\in\Sigma.
\]
Suppose also that the map
\[
G:=\pi_v\circ F:U\to v^\perp,
\qquad
F(u)=(u,\varphi(u)),
\]
has multiplicity bounded by $N_0$, that is,
\[
\#G^{-1}(\{y\})\le N_0
\qquad\text{for almost every }y\in v^\perp.
\]
Then there exists a constant $C\ge1$, depending only on $c_0$, on $\operatorname{Lip}(\varphi)$, on $M$, and on $N_0$, such that
\[
(\pi_v)_\#\mu\bigl(B_{v^\perp}(z,r)\bigr)
\le C r^2
\]
for every $z\in v^\perp$ and every $r>0$.
\end{lem}

\begin{proof}
Let
\[
F:U\to\Sigma,
\qquad
F(u)=(u,\varphi(u)).
\]
Since $\Sigma$ is a graph, $F$ is one-to-one and parametrizes $\Sigma$. Hence, for
\[
G=\pi_v\circ F,
\]
one has
\[
N(y,\pi_v,\Sigma)
=
\#G^{-1}(\{y\})
\]
for $\mathcal H^2$-a.e. $y\in v^\perp$, up to the harmless convention that both sides may be infinite. Therefore the assumed multiplicity bound for $G$ gives
\[
N(y,\pi_v,\Sigma)\le N_0
\]
for $\mathcal H^2$-a.e. $y\in v^\perp$.

Apply Lemma~\ref{lem:projective-area-formula-bound} to the graph $\Sigma$, the measure
\[
\mu=h\,\mathcal H^2\lfloor\Sigma,
\]
and the projection $\pi_v$. Since
\[
J_{\Sigma}\pi_v(x)=|n_\Sigma(x)\cdot v|
\]
for $\mathcal H^2$-a.e. $x\in\Sigma$, the transversality hypothesis gives
\[
J_{\Sigma}\pi_v(x)\ge c_0
\]
for $\mathcal H^2$-a.e. $x\in\Sigma$. The conclusion follows.
\end{proof}

\begin{cor}[transversal wave packet]\label{cor:surface-wave-packet-transverse}
Under the hypotheses of Lemma~\ref{lem:surface-projection-transverse}, let $\rho\in(0,1]$ and let $T$ be a tube of length comparable to $1$ and radius $\rho$, with axis parallel to $v$. Let $\phi_T$ be a function such that
\[
|\phi_T(x)|\le A\rho^{-1}
\left(1+\frac{\operatorname{dist}(x,T)}{\rho}\right)^{-N}
\]
for some constants $A\ge1$ and $N>2$. Then
\[
\int_{\mathbb R^3}|\phi_T(x)|^2\,\mathrm{d}\mu(x)
\le C A^2,
\]
where $C$ depends only on $N$, on $c_0$, on $\operatorname{Lip}(\varphi)$, on $M$, and on $N_0$.
\end{cor}

\begin{proof}
By Lemma~\ref{lem:surface-projection-transverse}, the projected measure $(\pi_v)_\#\mu$ satisfies a growth bound with exponent $\beta=2$. Corollary~\ref{cor:wave-packet-projection-frostman} gives
\[
\int_{\mathbb R^3}|\phi_T(x)|^2\,\mathrm{d}\mu(x)
\le
C A^2\rho^{2-2}=C A^2.
\]
\end{proof}

\begin{remark}\label{rem:surface-transverse}
The transversal regime produces uniform quadratic control because the orthogonal projection preserves area with a positive lower Jacobian bound and bounded multiplicity. This mechanism accounts for the lossless part of the projective argument. The loss in the quadratic model comes from the extreme tangential regime described in Section~\ref{sec:bad-geometry}.
\end{remark}
\section{Geometry of tangential concentration}\label{sec:bad-geometry}

This section proves the local geometric estimate responsible for tangential concentration in the nondegenerate quadratic model. We keep the notation fixed in Subsection~\ref{subsec:modelo-cuadratico-medida}.

For a direction $v\in S^2$, define the extreme strip by
\[
E_v^{\mathrm{ext}}
:=
\{x\in S_Q: |n_Q(x)\cdot v|\le \rho^{1/2}\},
\]
where $n_Q$ is the unit normal field of $S_Q$.

The main geometric point is that a tangent tube may capture surface mass of
order $\rho^{3/2}$ inside this strip.

\subsection{Quadratic normal and the scale of the extreme strip}\label{subsec:quadratic-tangency-geometry}

\begin{lem}[normal of the quadratic patch]\label{lem:quadratic-normal-linearization}
Let
\[
X(u_1,u_2)=(u_1,u_2,Q(u_1,u_2)).
\]
Then the unit normal field can be written as
\[
n_Q(u)
=
\frac{(-\lambda_1u_1,-\lambda_2u_2,1)}
{\sqrt{1+\lambda_1^2u_1^2+\lambda_2^2u_2^2}}.
\]
In particular, if $v_0=e_1$ is a tangent direction at the origin, then, for $u$ in a sufficiently small neighbourhood of the origin,
\[
|n_Q(u)\cdot v_0|\sim |u_1|,
\]
with constants depending only on $Q$ and on the size of the chart.
\end{lem}

\begin{proof}
One has
\[
\partial_{u_1}X=(1,0,\lambda_1u_1),
\qquad
\partial_{u_2}X=(0,1,\lambda_2u_2).
\]
Therefore,
\[
\partial_{u_1}X\times \partial_{u_2}X
=
(-\lambda_1u_1,-\lambda_2u_2,1),
\]
and normalization gives the formula for $n_Q$. If $v_0=e_1$, then
\[
|n_Q(u)\cdot v_0|
=
\frac{|\lambda_1u_1|}
{\sqrt{1+\lambda_1^2u_1^2+\lambda_2^2u_2^2}}.
\]
After shrinking the chart, the denominator is uniformly comparable to $1$. Since $\lambda_1\neq0$, the asserted comparison follows.
\end{proof}

\begin{lem}[extreme mass of a tangent tube]\label{lem:quadratic-extreme-strip-mass}
Let $v_0$ be a tangent direction to $S_Q$ at a point $x_0$ such that $\chi$ is positive in a neighbourhood of $x_0$. For $0<\rho\le \rho_0$, with $\rho_0$ depending only on the chart, on $Q$, and on the neighbourhood where $\chi$ is positive, there exists a tube $T_\rho$ of length comparable to $1$, transverse radius comparable to $\rho$, and direction parallel to $v_0$, such that
\[
\mu_Q(S_Q\cap T_\rho\cap E_{v_0}^{\mathrm{ext}})
\gtrsim
\rho^{3/2}.
\]
The implicit constant depends only on $Q$, on $\chi$ in the neighbourhood of $x_0$, and on the geometric constants in the definition of tube, but not on $\rho$.
\end{lem}

\begin{proof}
After a translation and a rigid rotation in the tangent plane, one may assume that $x_0=0$ and $v_0=e_1$. In local coordinates,
\[
X(u_1,u_2)=(u_1,u_2,Q(u_1,u_2)).
\]
By Lemma~\ref{lem:quadratic-normal-linearization}, if $|u_1|\le c\rho^{1/2}$ and $c>0$ is sufficiently small, then
\[
|n_Q(u)\cdot v_0|\le \rho^{1/2}.
\]

Let $T_\rho$ be a tube with direction $e_1$, transverse radius comparable to $\rho$, whose axis passes through the origin. For $|u|$ sufficiently small, the condition $X(u)\in T_\rho$ is guaranteed by
\[
|u_2|\le c\rho,
\qquad
|Q(u_1,u_2)|\le c\rho.
\]
Since $Q(u)=O(|u|^2)$, the restrictions
\[
|u_1|\le c\rho^{1/2},
\qquad
|u_2|\le c\rho
\]
imply $X(u)\in T_\rho$, after reducing $c$ if necessary.

Therefore,
\[
X\bigl(\{|u_1|\le c\rho^{1/2},\ |u_2|\le c\rho\}\bigr)
\subset
S_Q\cap T_\rho\cap E_{v_0}^{\mathrm{ext}}.
\]
The surface Jacobian of $X$ is comparable to $1$ in the chart, and $\chi$ is bounded below by a positive constant in the neighbourhood under consideration. We conclude that
\[
\mu_Q(S_Q\cap T_\rho\cap E_{v_0}^{\mathrm{ext}})
\gtrsim
\rho^{1/2}\rho
=
\rho^{3/2}.
\]
\end{proof}

\begin{lem}[upper bound for extreme mass in a tangent tube]\label{lem:quadratic-extreme-strip-mass-upper}
Let
\[
S_Q=\{(u_1,u_2,Q(u_1,u_2)):u\in U\},
\qquad
Q(u_1,u_2)=\frac12(\lambda_1u_1^2+\lambda_2u_2^2),
\]
with $\lambda_1\lambda_2>0$. Let
\[
\mu_Q=\chi\,\mathcal H^2\lfloor S_Q,
\]
where $\chi\in L^\infty(S_Q)$ is compactly supported in a sufficiently small chart. Suppose, after a translation and a rigid rotation in the tangent plane, that the base point is the origin and that the tangent direction under consideration is $v_0=e_1$.

Then there exist $\rho_0\in(0,1]$ and $C\ge1$, depending only on $Q$, on the chart, on $\|\chi\|_{L^\infty}$, and on the geometric constants in the definition of tube, such that for every $0<\rho\le\rho_0$ and every tube $T_\rho$ of length comparable to $1$, transverse radius comparable to $\rho$, and direction parallel to $v_0$, one has
\[
\mu_Q(S_Q\cap T_\rho\cap E_{v_0}^{\mathrm{ext}})
\le
C\rho^{3/2}.
\]
\end{lem}

\begin{proof}
In the local chart we write
\[
X(u_1,u_2)=(u_1,u_2,Q(u_1,u_2)).
\]
By Lemma~\ref{lem:quadratic-normal-linearization}, after shrinking the chart if necessary,
\[
|n_Q(u)\cdot v_0|\sim |u_1|.
\]
Therefore there exists a constant $C_1\ge1$ such that
\[
X(u)\in E_{v_0}^{\mathrm{ext}}
\quad\Longrightarrow\quad
|u_1|\le C_1\rho^{1/2}.
\]

Let $T_\rho$ be a tube with direction $e_1$. Its axis can be written as
\[
\ell=\{(t,a,b):t\in\mathbb{R}\}
\]
for some $a,b\in\mathbb{R}$. If $X(u_1,u_2)\in T_\rho$, then the distance from $(u_2,Q(u_1,u_2))$ to the point $(a,b)$ in the transverse plane $(x_2,x_3)$ is $\lesssim\rho$. In particular,
\[
|u_2-a|\lesssim \rho.
\]
Thus, for each fixed $u_1$, the set of values $u_2$ such that $X(u_1,u_2)\in T_\rho$ has length $\lesssim\rho$.

Combining this restriction with membership in the extreme strip, the parameter set
\[
\{u\in U:X(u)\in T_\rho\cap E_{v_0}^{\mathrm{ext}}\}
\]
is contained in a set of area $\lesssim \rho^{1/2}\rho=\rho^{3/2}$ in the parameter plane. The surface Jacobian of $X$ is uniformly bounded in the chart, and $\chi$ is bounded above by $\|\chi\|_{L^\infty}$. Therefore,
\[
\mu_Q(S_Q\cap T_\rho\cap E_{v_0}^{\mathrm{ext}})
\le
C\rho^{3/2}.
\]
\end{proof}

\begin{remark}[scope of the extreme tubular bound]\label{rem:quadratic-extreme-strip-mass-upper-scope}
Lemma~\ref{lem:quadratic-extreme-strip-mass-upper} is an individual geometric bound for tangent tubes in the quadratic chart. In particular, it identifies the maximal scale
\[
\mu_Q(S_Q\cap T_\rho\cap E_{v_0}^{\mathrm{ext}})
\lesssim \rho^{3/2}.
\]
This individual estimate does not imply a quadratic bound for a full family of wave packets. To pass from the individual tubular bound to an estimate of the form
\[
\int_{E_{v_0}^{\mathrm{ext}}}
\left|
\sum_T c_T\phi_T(x)
\right|^2\,\mathrm{d}\mu_Q(x)
\lesssim
\rho^{-1/2}\sum_T |c_T|^2
\]
one additionally needs an almost orthogonality, Schur, or overlap mechanism for the cross terms of the tubular family.
\end{remark}

\begin{remark}[local extreme scale]\label{rem:quadratic-extreme-local-scale}
Lemmas~\ref{lem:quadratic-extreme-strip-mass} and~\ref{lem:quadratic-extreme-strip-mass-upper} identify the extreme geometric scale in the tangent model: the maximal mass of the extreme strip inside a tangent tube is comparable to $\rho^{3/2}$. If a normalized wave packet is concentrated in $T_\rho$ and has size
\[
|\phi_T|\sim \rho^{-1}
\]
in the tube, then its effective quadratic size is
\[
|\phi_T|^2\sim \rho^{-2}.
\]
The extreme mass $\rho^{3/2}$ produces the quadratic cost
\[
\rho^{-2}\rho^{3/2}=\rho^{-1/2},
\]
and therefore the norm loss
\[
\rho^{-1/4}=R^{1/8}.
\]
This gives the residual scale of the quadratic model with positive surface density. The tangent packet test below shows that this scale is optimal.
\end{remark}

\subsection{Uniform elliptic form of the extreme strip}\label{subsec:elliptic-uniform-extreme-geometry}

In this subsection we fix a compact chart $U_0\Subset U$ such that
\[
\operatorname{supp}\chi\subset S_{Q,0}:=X(U_0),
\qquad
X(u_1,u_2)=(u_1,u_2,Q(u_1,u_2)).
\]
All implicit constants may depend on $Q$, on $U_0$, on $\|\chi\|_{L^\infty}$, and on the geometric constants of the tubes, but not on $\rho$, on $v$, or on the tube under consideration.

\begin{defn}[relevant directions]\label{def:quadratic-relevant-directions}
We say that $v\in S^2$ is a relevant direction for the patch if
\[
Z_v
:=
\{u\in U_0: n_Q(X(u))\cdot v=0\}
\]
is nonempty. For such $v$, the extreme strip is defined by
\[
E_v^{\mathrm{ext}}
:=
\{X(u):u\in U_0,\ |n_Q(X(u))\cdot v|\le \rho^{1/2}\}.
\]
\end{defn}

\begin{lem}[uniform transversality of the tangency function]\label{lem:quadratic-relevant-direction-transversality}
There exists a constant $c_Q>0$, depending only on $Q$ and on $U_0$, such that for every relevant direction $v\in S^2$ and every $u\in Z_v$,
\[
\bigl|\nabla_u(n_Q(X(u))\cdot v)\bigr|\ge c_Q.
\]
\end{lem}

\begin{proof}
Set
\[
h_v(u):=n_Q(X(u))\cdot v.
\]
If $u\in Z_v$, then $v\in T_{X(u)}S_Q$. Let $a\in T_uU$ be the unique vector such that
\[
\mathrm{d}X_u a=v.
\]
Since $\mathrm{d}X_u:T_uU\to T_{X(u)}S_Q$ is an isomorphism and $U_0\Subset U$, there exist constants $0<c_0\le C_0<\infty$, depending only on $Q$ and on $U_0$, such that
\[
c_0\le |a|\le C_0.
\]

For $b\in T_uU$, up to the sign determined by the chosen orientation of the normal, one has
\[
\mathrm{d} h_v(u)[b]
=
\mathrm{d}(n_Q\circ X)_u[b]\cdot \mathrm{d}X_u a
=
-\mathrm{II}_u(b,a).
\]
Therefore, the norm of $\nabla h_v(u)$ is uniformly comparable to the norm of the functional
\[
b\mapsto \mathrm{II}_u(b,a).
\]
On the quadratic graph,
\[
\mathrm{II}_u(b,a)
=
\frac{\lambda_1 b_1a_1+\lambda_2 b_2a_2}
{\sqrt{1+\lambda_1^2u_1^2+\lambda_2^2u_2^2}},
\]
up to the same orientation sign. Since $\lambda_1\lambda_2>0$, the principal coefficients have modulus uniformly separated from zero and the denominator is uniformly bounded above and below on $U_0$. Consequently,
\[
\sup_{|b|=1}|\mathrm{II}_u(b,a)|\ge c|a|\ge cc_0,
\]
with $c>0$ depending only on $Q$ and on $U_0$. This gives the asserted lower bound.
\end{proof}

\begin{lem}[uniform normal form near $E_v^{\mathrm{ext}}$]\label{lem:quadratic-uniform-normal-form}
There exist constants $\rho_0\in(0,1]$, $C\ge1$, and a finite family of coordinate charts, depending only on $Q$ and on $U_0$, with the following property.

For every $0<\rho\le\rho_0$, every relevant direction $v$, and every point of $E_v^{\mathrm{ext}}$, there is one of these local charts $(s,t)$ in the parameter domain such that, in that chart,
\[
s=h_v(u),
\qquad
h_v(u):=n_Q(X(u))\cdot v,
\]
and
\[
E_v^{\mathrm{ext}}\subset \{|s|\le C\rho^{1/2}\}.
\]
The Jacobians of these charts and of their inverses are uniformly bounded by constants depending only on $Q$ and on $U_0$.
\end{lem}

\begin{proof}
The family of pairs
\[
\mathcal K_0
:=
\{(u,v)\in U_0\times S^2:h_v(u)=0\}
\]
is compact. By Lemma~\ref{lem:quadratic-relevant-direction-transversality}, $|\nabla h_v(u)|\ge c_Q$ on $\mathcal K_0$. The implicit function theorem with uniform constants, applied on a finite cover of $\mathcal K_0$, gives local charts $(s,t)$ with $s=h_v(u)$ and uniformly controlled Jacobians.

After reducing $\rho_0$ if necessary, every point with $|h_v(u)|\le \rho^{1/2}$ lies in the charts of this cover. In those charts,
\[
X(u)\in E_v^{\mathrm{ext}}
\quad\Longrightarrow\quad
|h_v(u)|\le \rho^{1/2},
\]
and, since $s=h_v(u)$ up to the uniform normalization of the chart, one obtains
\[
|s|\le C\rho^{1/2}.
\]
\end{proof}

\begin{lem}[projective nondegeneracy of the elliptic tangency curve]\label{lem:elliptic-tangency-curve-projection}
Let
\[
S_Q=\{(u_1,u_2,Q(u_1,u_2)):u\in U\},
\qquad
Q(u_1,u_2)=\frac12(\lambda_1u_1^2+\lambda_2u_2^2),
\qquad
\lambda_1\lambda_2>0.
\]
Let $U_0\Subset U$. For $v\in S^2$, define
\[
h_v(u):=n_Q(X(u))\cdot v,
\qquad
Z_v:=\{u\in U_0:h_v(u)=0\}.
\]
Then there exists $c>0$, depending only on $Q$ and on $U_0$, with the following property. For every relevant direction $v$, every regular arc of $Z_v$, and every arc-length parametrization $\gamma:I\to Z_v$,
\[
\left|
\frac{\mathrm{d}}{\mathrm{d}t}
\pi_v(X(\gamma(t)))
\right|
\ge c
\qquad
\text{for every }t\in I.
\]
\end{lem}

\begin{proof}
Let $u=\gamma(t)$ and write $w:=\gamma'(t)$. Since $\gamma$ is parametrized by arc length in the domain, $|w|=1$. Moreover, $h_v(\gamma(t))=0$, and hence
\[
\mathrm{d} h_v(u)[w]=0.
\]
Since $h_v(u)=0$, one has $v\in T_{X(u)}S_Q$. Let $a\in T_uU$ be the unique vector such that
\[
\mathrm{d}X_u a=v.
\]

If
\[
\pi_v(\mathrm{d}X_u w)=0,
\]
then $\mathrm{d}X_u w$ is parallel to $v$. Therefore there exists $\alpha\in\mathbb{R}$ such that
\[
\mathrm{d}X_u w=\alpha v=\alpha\,\mathrm{d}X_u a.
\]
The injectivity of $\mathrm{d}X_u$ gives $w=\alpha a$.

On the other hand,
\[
0=\mathrm{d} h_v(u)[w]
=
\mathrm{d}(n_Q\circ X)_u[w]\cdot \mathrm{d}X_u a
=
-\mathrm{II}_u(w,a),
\]
with the sign depending on the orientation of the normal. Since $w=\alpha a$,
\[
0=\alpha\,\mathrm{II}_u(a,a).
\]
If $\alpha=0$, then $w=0$, a contradiction. Hence $\alpha\neq0$ and $\mathrm{II}_u(a,a)=0$.

For the quadratic graph,
\[
\mathrm{II}_u(b,b)
=
\frac{\lambda_1 b_1^2+\lambda_2 b_2^2}
{\sqrt{1+\lambda_1^2u_1^2+\lambda_2^2u_2^2}},
\]
up to the orientation sign. Since $\lambda_1\lambda_2>0$, this quadratic form is definite. Thus $\mathrm{II}_u(a,a)=0$ implies $a=0$, and then $v=\mathrm{d}X_u a=0$, contradicting $v\in S^2$. Therefore
\[
\pi_v(\mathrm{d}X_u w)\neq0
\]
for every admissible triple $(u,v,w)$.

The uniform lower bound follows by compactness of the set
\[
\mathcal K
:=
\{(u,v,w):u\in U_0,\ v\in S^2,\ h_v(u)=0,\ |w|=1,\ \mathrm{d} h_v(u)[w]=0\}.
\]
The function $(u,v,w)\mapsto |\pi_v(\mathrm{d}X_u w)|$ is continuous and does not vanish on $\mathcal K$. This gives the constant $c>0$ and concludes the proof.
\end{proof}

\begin{lem}[elliptic tubular sublevel estimate]\label{lem:elliptic-tubular-sublevel}
Let
\[
S_Q=\{(u_1,u_2,Q(u_1,u_2)):u\in U\},
\qquad
Q(u_1,u_2)=\frac12(\lambda_1u_1^2+\lambda_2u_2^2),
\qquad
\lambda_1\lambda_2>0.
\]
Let
\[
\mu_Q=\chi\,\mathcal H^2\lfloor S_Q,
\]
where $\chi\in L^\infty(S_Q)$ is compactly supported in $S_{Q,0}=X(U_0)$, with $U_0\Subset U$. For $v\in S^2$, define
\[
E_v^{\mathrm{ext}}
:=
\{x\in S_{Q,0}: |n_Q(x)\cdot v|\le \rho^{1/2}\}.
\]
Then there exist constants $\rho_0\in(0,1]$, $\sigma_0\in(0,1]$, and $C\ge1$, depending only on $Q$, on $U_0$, on $\|\chi\|_{L^\infty}$, and on the geometric constants in the definition of tube, such that for every $0<\rho\le\rho_0$, every relevant direction $v$, every tube $T$ of length comparable to $1$, radius comparable to $\rho$, and direction parallel to $v$, and every $\sigma$ with
\[
\rho\le \sigma\le \sigma_0,
\]
one has
\[
\mu_Q\bigl(E_v^{\mathrm{ext}}\cap N_\sigma(T)\bigr)
\le
C\rho^{1/2}\sigma.
\]
\end{lem}

\begin{proof}
Set
\[
h_v(u):=n_Q(X(u))\cdot v.
\]
By Lemma~\ref{lem:quadratic-uniform-normal-form}, the extreme strip is covered by a finite family of charts $(s,t)$, uniform in $v$, in which
\[
|s|\le C\rho^{1/2}
\]
and the Jacobians are uniformly controlled.

Fix one of these charts. It is enough to prove that, in this chart,
\[
\bigl|\{(s,t): |s|\le C\rho^{1/2},\ X(s,t)\in N_\sigma(T)\}\bigr|
\lesssim
\rho^{1/2}\sigma.
\]
For each fixed $s$, we estimate the length of the corresponding set of parameters $t$.

Let $\ell_T$ be the axis of $T$. Since $T$ has direction parallel to $v$, there exists $y_T\in v^\perp$ such that
\[
X(s,t)\in N_\sigma(T)
\quad\Longrightarrow\quad
|\pi_v(X(s,t))-y_T|\le C_T\sigma,
\]
where $C_T$ depends only on the geometric convention for tubes.

For $s=0$, the curve $t\mapsto X(0,t)$ parametrizes an arc of $Z_v$. By Lemma~\ref{lem:elliptic-tangency-curve-projection}, and after refining the finite family of charts if necessary, there exists a unit linear component $e_{v,\alpha}\in v^\perp$, fixed in each chart of the cover, such that
\[
\left|
\frac{\partial}{\partial t}
\bigl(\pi_v(X(0,t))\cdot e_{v,\alpha}\bigr)
\right|
\ge c_1
\]
in that chart. Here $c_1>0$ depends only on $Q$ and on $U_0$. By continuity, after reducing $\rho_0$ if necessary, the same estimate holds for $|s|\le C\rho^{1/2}$ with $c_1/2$ in place of $c_1$:
\[
\left|
\frac{\partial}{\partial t}
\bigl(\pi_v(X(s,t))\cdot e_{v,\alpha}\bigr)
\right|
\ge \frac{c_1}{2}.
\]
After reducing each coordinate interval in the finite cover, one may also assume that this derivative has constant sign. Hence the scalar function
\[
t\mapsto \pi_v(X(s,t))\cdot e_{v,\alpha}
\]
is quantitatively monotone, uniformly in $s$, $v$, and $\rho$. Consequently,
\[
\bigl|\{t: |\pi_v(X(s,t))-y_T|\le C_T\sigma\}\bigr|
\lesssim \sigma.
\]

Integrating this bound in $s$ over an interval of length $\lesssim\rho^{1/2}$, and summing over the finite family of charts, gives
\[
\mathcal H^2\bigl(E_v^{\mathrm{ext}}\cap N_\sigma(T)\bigr)
\lesssim
\rho^{1/2}\sigma.
\]
Finally,
\[
\mu_Q\bigl(E_v^{\mathrm{ext}}\cap N_\sigma(T)\bigr)
\le
\|\chi\|_{L^\infty}
\mathcal H^2\bigl(E_v^{\mathrm{ext}}\cap N_\sigma(T)\bigr),
\]
which concludes the proof.
\end{proof}

\begin{lem}[extreme mass with tubular tails]\label{lem:quadratic-extreme-tail-mass}
Let $S_Q$, $\mu_Q$, $v_0=e_1$, and the quadratic chart be as in Lemma~\ref{lem:quadratic-extreme-strip-mass-upper}. Let $T_\rho$ be a tube of length comparable to $1$, transverse radius comparable to $\rho$, and direction parallel to $v_0$. Then, for every $N>2$, there exist $\rho_0\in(0,1]$ and $C_N\ge1$, depending only on $N$, on $Q$, on the chart, on $\|\chi\|_{L^\infty}$, and on the geometric constants in the definition of tube, such that for every $0<\rho\le\rho_0$,
\[
\int_{E_{v_0}^{\mathrm{ext}}}
\left(
1+\frac{\operatorname{dist}(x,T_\rho)}{\rho}
\right)^{-N}
\,\mathrm d\mu_Q(x)
\le
C_N\rho^{3/2}.
\]
\end{lem}

\begin{proof}
In the local chart we write
\[
X(u_1,u_2)=(u_1,u_2,Q(u_1,u_2)).
\]
By Lemma~\ref{lem:quadratic-normal-linearization}, after shrinking the chart if necessary,
\[
X(u)\in E_{v_0}^{\mathrm{ext}}
\quad\Longrightarrow\quad
|u_1|\le C\rho^{1/2}.
\]
Let the axis of $T_\rho$ be of the form
\[
\ell=\{(t,a,b):t\in\mathbb R\}.
\]
If $\operatorname{dist}(X(u),T_\rho)\lesssim 2^k\rho$, then, in particular,
\[
|u_2-a|\lesssim 2^k\rho
\]
as long as the chart remains fixed. Therefore,
\[
\mu_Q\Bigl(
\{x\in E_{v_0}^{\mathrm{ext}}:\operatorname{dist}(x,T_\rho)\lesssim 2^k\rho\}
\Bigr)
\lesssim
\rho^{1/2}\,2^k\rho
=
2^k\rho^{3/2}.
\]
The implicit constant depends only on $Q$, on the chart, on $\|\chi\|_{L^\infty}$, and on the geometric constants of the tubes.

Decomposing into dyadic annuli,
\[
A_0=\{x\in E_{v_0}^{\mathrm{ext}}:\operatorname{dist}(x,T_\rho)\lesssim \rho\},
\]
\[
A_k=\{x\in E_{v_0}^{\mathrm{ext}}:2^{k-1}\rho\lesssim\operatorname{dist}(x,T_\rho)\lesssim 2^k\rho\},
\qquad k\ge1,
\]
one obtains
\[
\int_{E_{v_0}^{\mathrm{ext}}}
\left(
1+\frac{\operatorname{dist}(x,T_\rho)}{\rho}
\right)^{-N}
\,\mathrm d\mu_Q(x)
\lesssim
\sum_{k\ge0}2^{-Nk}\,2^k\rho^{3/2}.
\]
The series converges under the stated assumption $N>2$. This proves the bound.
\end{proof}

\subsection{Local functional consequence}\label{subsec:quadratic-extreme-functional-cost}

\begin{prop}[tangent packet test]\label{prop:quadratic-extreme-wave-packet-cost}
Under the hypotheses of Lemma~\ref{lem:quadratic-extreme-strip-mass}, suppose that there exists a test wave packet $\phi_T$ associated with the tangent tube $T_\rho$ such that
\[
\|\phi_T\|_{L^2(\mathbb{R}^3)}\sim1
\]
and
\[
|\phi_T(x)|\gtrsim \rho^{-1}
\]
on $S_Q\cap T_\rho\cap E_{v_0}^{\mathrm{ext}}$. Then
\[
\|\phi_T\|_{L^2(\mathrm{d}\mu_Q)}
\gtrsim
\rho^{-1/4}
\|\phi_T\|_{L^2(\mathbb{R}^3)}
=
R^{1/8}
\|\phi_T\|_{L^2(\mathbb{R}^3)}.
\]
\end{prop}

\begin{proof}
By the pointwise lower bound and Lemma~\ref{lem:quadratic-extreme-strip-mass},
\[
\|\phi_T\|_{L^2(\mathrm{d}\mu_Q)}^2
\ge
\int_{S_Q\cap T_\rho\cap E_{v_0}^{\mathrm{ext}}}
|\phi_T(x)|^2\,\mathrm{d}\mu_Q(x)
\gtrsim
\rho^{-2}\rho^{3/2}
=
\rho^{-1/2}.
\]
Taking square roots gives
\[
\|\phi_T\|_{L^2(\mathrm{d}\mu_Q)}
\gtrsim
\rho^{-1/4}.
\]
Since $\|\phi_T\|_{L^2(\mathbb{R}^3)}\sim1$ and $\rho=R^{-1/2}$, the claim follows.
\end{proof}

\begin{cor}[obstruction to a uniform bound]\label{cor:quadratic-no-uniform-bound}
In the quadratic model with positive surface density, there cannot exist a constant $C$ independent of $R$ such that
\[
\|\mathcal G_R f\|_{L^2(\mathrm{d}\mu_Q)}
\le
C\|f\|_{L^2(\mathbb{R}^3)}
\]
for the class of functions containing the model packets of Lemma~\ref{lem:model-wave-packet-in-box}.
\end{cor}

\begin{proof}
Let $x_0\in\operatorname{supp}\chi$ be a tangency point in whose surface neighbourhood $\chi$ is bounded below by a positive constant, and let $v_0\in T_{x_0}S_Q$ be a unit tangent direction as in Proposition~\ref{prop:quadratic-extreme-wave-packet-cost}. Choose a frequency box $\Theta_0$ of the parabolic decomposition whose central physical direction is $v_0$, up to the angular error allowed by the aperture $R^{-1/2}$. Let $\phi_T$ be the model packet given by Lemma~\ref{lem:model-wave-packet-in-box}, associated with $\Theta_0$ and adapted to the tangent tube $T$ used in that proposition.

Taking $f=f_{\Theta_0}=\phi_T$ and $f_\Theta=0$ for $\Theta\neq\Theta_0$, the quantity $\mathcal G_R f$ agrees with $|\phi_T|$ up to uniform constants coming from finite overlap. Proposition~\ref{prop:quadratic-extreme-wave-packet-cost} gives
\[
\|\mathcal G_R f\|_{L^2(\mathrm d\mu_Q)}
\gtrsim
\rho^{-1/4}.
\]
On the other hand, by Lemma~\ref{lem:model-wave-packet-in-box},
\[
\|f\|_{L^2(\mathbb R^3)}
=
\|\phi_T\|_{L^2(\mathbb R^3)}
\sim 1.
\]
Since $\rho=R^{-1/2}$, one has
\[
\rho^{-1/4}=R^{1/8}.
\]
Therefore, a uniform bound
\[
\|\mathcal G_R f\|_{L^2(\mathrm d\mu_Q)}
\le C\|f\|_{L^2(\mathbb R^3)}
\]
with $C$ independent of $R$ is impossible in this class of tests.
\end{proof}

\subsection{Refined reading: transversal, extreme, and intermediate}\label{subsec:quadratic-refined-decomposition}

The preceding geometry is compatible with an upper estimate with loss $R^{1/8}$, but excludes a bound uniform in $R$. The decomposition compatible with the quadratic model separates the contribution into three regimes:
\[
\text{transversal},
\qquad
\text{extreme},
\qquad
\text{intermediate}.
\]

The transversal regime corresponds to directions $v$ for which $|n_Q(x)\cdot v|$ is bounded below in the relevant region. There the projection onto $v^\perp$ has effective two-dimensional behaviour and the expected quadratic cost is uniform.

The extreme regime corresponds to
\[
|n_Q(x)\cdot v|\le \rho^{1/2}.
\]
In this regime, Lemma~\ref{lem:quadratic-extreme-strip-mass} shows that a tangent tube may capture mass $\rho^{3/2}$, and therefore produces the optimal quadratic cost $\rho^{-1/2}$.

The intermediate regime corresponds to
\[
\rho^{1/2}<|n_Q(x)\cdot v|<\tau,
\]
with $\tau\in(0,1)$ fixed. This regime is distinct from the extreme strip. Its treatment requires additional Gramian, oscillatory, or internal almost-orthogonality estimates. A uniform bound for the intermediate strip therefore requires a separate
input; it does not follow from the extreme tubular mass.

\begin{prop}[geometric assembly of costs]\label{prop:quadratic-geometric-cost-assembly}
Suppose that, for the angular decomposition under consideration, the transversal contribution satisfies a quadratic bound with constant $C_{\mathrm{tr}}$ independent of $R$, and that the remaining nontransversal contribution is controlled with quadratic cost
\[
C_{\mathrm{nt}}\rho^{-1/2},
\]
where $C_{\mathrm{nt}}$ is independent of $R$. Then the total quadratic cost is
\[
\lesssim
\rho^{-1/2},
\]
and the corresponding norm cost is
\[
\lesssim
\rho^{-1/4}=R^{1/8}.
\]
\end{prop}

\begin{proof}
The claim follows from
\[
A X+B Y\le \max\{A,B\}(X+Y),
\qquad X,Y\ge0,
\]
applied with $A=C_{\mathrm{tr}}$ and $B=C_{\mathrm{nt}}\rho^{-1/2}$. Since $0<\rho\le1$,
\[
\max\{C_{\mathrm{tr}},C_{\mathrm{nt}}\rho^{-1/2}\}
\le
C\rho^{-1/2},
\]
with $C$ depending only on $C_{\mathrm{tr}}$ and $C_{\mathrm{nt}}$. Taking square roots converts the quadratic cost $\rho^{-1/2}$ into the norm cost $\rho^{-1/4}$.
\end{proof}

\section{Analytic reduction to the optimal scale}\label{sec:reduccion-wave-packets}

This section assembles the diagonal estimate in the elliptic quadratic model fixed in Subsection~\ref{subsec:modelo-cuadratico-medida}. We keep the notation fixed there. The local positivity of $\chi$ is used only for the lower-bound test, not for the upper bound.

The analytic reduction separates two quadratic inputs: a transversal
contribution of uniform cost and a nontransversal contribution of cost
$\rho^{-1/2}$. After taking square roots, the second cost produces the loss
\[
\rho^{-1/4}=R^{1/8}.
\]
The geometric verification of the extreme mass is carried out in
Section~\ref{sec:bad-geometry}; here we record how that scale enters the
nontransversal control, the analytic assembly, and the lower bound.

\begin{defn}[transversal--nontransversal partition]\label{def:quadratic-transversal-extreme-partition}
Let $\tau_0\in(0,1)$ be fixed, independently of $R$. For each frequency box $\Theta$, let $v_\Theta\in S^2$ be the physical direction associated with the wave packets dual to $\Theta$. Define
\[
\mathbb T_{\tau_0}
:=
\left\{
\Theta:
\inf_{x\in \operatorname{supp}\chi}
|n_Q(x)\cdot v_\Theta|
\ge \tau_0
\right\},
\]
and
\[
\mathbb E_{\tau_0}
:=
\{\Theta\}\setminus \mathbb T_{\tau_0}.
\]
The subfamily $\mathbb T_{\tau_0}$ will be treated by the transversal projective mechanism. The subfamily $\mathbb E_{\tau_0}$ contains every regime not covered by this uniform transversality and is estimated with the admissible nontransversal quadratic cost $\rho^{-1/2}$. Within this second subfamily, the extreme strip is controlled by tubular Schur and the complement by degenerate projection.
\end{defn}

\begin{prop}[quadratic assembly of two contributions]\label{prop:quadratic-two-piece-assembly}
Let $\mu$ be a positive Borel measure on $\mathbb{R}^{3}$. Suppose that the family of frequency boxes admits a disjoint decomposition
\[
\{\Theta\}=\mathbb{T}\,\dot\cup\,\mathbb{E},
\]
and that there exist constants $A,B\ge0$ such that
\[
\int_{\mathbb{R}^{3}}
\sum_{\Theta\in\mathbb{T}} |f_{\Theta}(x)|^{2}\,\mathrm{d}\mu(x)
\le
A
\sum_{\Theta\in\mathbb{T}}
\|f_{\Theta}\|_{L^{2}(\mathbb{R}^{3})}^{2}
\]
and
\[
\int_{\mathbb{R}^{3}}
\sum_{\Theta\in\mathbb{E}} |f_{\Theta}(x)|^{2}\,\mathrm{d}\mu(x)
\le
B
\sum_{\Theta\in\mathbb{E}}
\|f_{\Theta}\|_{L^{2}(\mathbb{R}^{3})}^{2}.
\]
Then
\[
\int_{\mathbb{R}^{3}}
\sum_{\Theta}|f_{\Theta}(x)|^{2}\,\mathrm{d}\mu(x)
\le
\max\{A,B\}
\sum_{\Theta}
\|f_{\Theta}\|_{L^{2}(\mathbb{R}^{3})}^{2}.
\]
\end{prop}

\begin{proof}
By the disjoint decomposition,
\[
\sum_{\Theta}|f_{\Theta}(x)|^{2}
=
\sum_{\Theta\in\mathbb{T}}|f_{\Theta}(x)|^{2}
+
\sum_{\Theta\in\mathbb{E}}|f_{\Theta}(x)|^{2}.
\]
Integrating with respect to $\mu$ and applying the two hypotheses gives
\[
\int_{\mathbb{R}^{3}}
\sum_{\Theta}|f_{\Theta}(x)|^{2}\,\mathrm{d}\mu(x)
\le
A
\sum_{\Theta\in\mathbb{T}}
\|f_{\Theta}\|_{L^{2}(\mathbb{R}^{3})}^{2}
+
B
\sum_{\Theta\in\mathbb{E}}
\|f_{\Theta}\|_{L^{2}(\mathbb{R}^{3})}^{2}.
\]
The conclusion follows from $AX+BY\le \max\{A,B\}(X+Y)$ for $X,Y\ge0$.
\end{proof}

\begin{remark}[scope of the extreme family bound]\label{rem:quadratic-extreme-family-scope}
The individual tubular bound
\[
\mu_Q(S_Q\cap T_\rho\cap E_v^{\mathrm{ext}})
\lesssim \rho^{3/2}
\]
from Lemma~\ref{lem:quadratic-extreme-strip-mass-upper} identifies the correct geometric scale per tube, but does not by itself imply a quadratic inequality for a sum of wave packets.

The passage from individual tubular mass to a family bound is carried out by Proposition~\ref{prop:quadratic-extreme-family-schur}. In Theorem~\ref{thm:quadratic-diagonal-sharp}, this bound is used box by box through Proposition~\ref{prop:quadratic-nontransversal-box-cost}. The extreme strip contributes the admissible quadratic cost $\rho^{-1/2}$; the complement is controlled by degenerate projection with the same cost.

In the elliptic model, the Schur hypotheses do not remain as additional assumptions: they are verified in
Corollary~\ref{cor:quadratic-extreme-schur-parallel}.

Proposition~\ref{prop:quadratic-extreme-family-schur} provides the abstract functional Schur step. The geometric verification of its hypotheses for the tubular families appearing in the standard decomposition is carried out in Corollary~\ref{cor:quadratic-extreme-schur-parallel}, using the pointwise overlap of tails from Lemma~\ref{lem:parallel-tube-tail-overlap} and the uniform tail mass bound from Section~\ref{sec:bad-geometry}.
\end{remark}

\begin{prop}[extreme family bound under tubular Schur]\label{prop:quadratic-extreme-family-schur}
Let $S_Q$, $\mu_Q$, and $\rho$ be as in Section~\ref{sec:bad-geometry}. Fix a tangent direction $v_0$ in a quadratic chart where Lemma~\ref{lem:quadratic-extreme-strip-mass-upper} holds. Let $\mathbb{T}_{v_0}$ be a family of tubes of length comparable to $1$, transverse radius comparable to $\rho$, and direction parallel to $v_0$.

Suppose that for each $T\in\mathbb{T}_{v_0}$ one has a function $\phi_T$ normalized by
\[
\|\phi_T\|_{L^2(\mathbb{R}^3)}\sim 1
\]
and spatially adapted to $T$ in the sense that, for some $N>10$ and some constant $A\ge1$,
\[
|\phi_T(x)|
\le
A\rho^{-1}
\left(
1+\frac{\operatorname{dist}(x,T)}{\rho}
\right)^{-N}
\qquad
\text{for every }x\in\mathbb{R}^3.
\]
Assume moreover that the tubular family satisfies the extreme Schur control
\[
\sup_{T\in\mathbb{T}_{v_0}}
\sum_{T'\in\mathbb{T}_{v_0}}
\left|
\int_{E_{v_0}^{\mathrm{ext}}}
\phi_T(x)\overline{\phi_{T'}(x)}\,\mathrm{d}\mu_Q(x)
\right|
\le
C_{\mathrm{Sch}}\rho^{-1/2},
\]
and symmetrically
\[
\sup_{T'\in\mathbb{T}_{v_0}}
\sum_{T\in\mathbb{T}_{v_0}}
\left|
\int_{E_{v_0}^{\mathrm{ext}}}
\phi_T(x)\overline{\phi_{T'}(x)}\,\mathrm{d}\mu_Q(x)
\right|
\le
C_{\mathrm{Sch}}\rho^{-1/2}.
\]
Then, for every finitely supported set of coefficients $\{c_T\}_{T\in\mathbb{T}_{v_0}}$,
\[
\int_{E_{v_0}^{\mathrm{ext}}}
\left|
\sum_{T\in\mathbb{T}_{v_0}} c_T\phi_T(x)
\right|^2
\,\mathrm{d}\mu_Q(x)
\le
C_{\mathrm{Sch}}\rho^{-1/2}
\sum_{T\in\mathbb{T}_{v_0}} |c_T|^2.
\]
\end{prop}

\begin{proof}
Define the matrix
\[
K(T,T')
:=
\int_{E_{v_0}^{\mathrm{ext}}}
\phi_T(x)\overline{\phi_{T'}(x)}\,\mathrm{d}\mu_Q(x).
\]
Then
\[
\int_{E_{v_0}^{\mathrm{ext}}}
\left|
\sum_{T\in\mathbb{T}_{v_0}} c_T\phi_T(x)
\right|^2
\,\mathrm{d}\mu_Q(x)
=
\sum_{T,T'\in\mathbb{T}_{v_0}}
c_T\overline{c_{T'}}K(T,T').
\]
By the discrete Schur criterion applied to the matrix $K$, the two hypotheses
\[
\sup_{T}
\sum_{T'} |K(T,T')|
\le
C_{\mathrm{Sch}}\rho^{-1/2},
\qquad
\sup_{T'}
\sum_T |K(T,T')|
\le
C_{\mathrm{Sch}}\rho^{-1/2}
\]
imply that the operator associated with $K$ is bounded on $\ell^2(\mathbb{T}_{v_0})$ with norm at most $C_{\mathrm{Sch}}\rho^{-1/2}$. Consequently,
\[
\left|
\sum_{T,T'\in\mathbb{T}_{v_0}}
c_T\overline{c_{T'}}K(T,T')
\right|
\le
C_{\mathrm{Sch}}\rho^{-1/2}
\sum_{T\in\mathbb{T}_{v_0}} |c_T|^2.
\]
This proves the claim.
\end{proof}

\begin{lem}[pointwise overlap of anisotropic tails]\label{lem:parallel-tube-tail-overlap}
Let $\rho\in(0,1]$, let $v\in S^2$, and fix an orthonormal basis $(e_1,e_2,v)$. For $m=(m_1,m_2)\in\mathbb Z^2$ and $j\in\mathbb Z$, define
\[
w_{m,j}(x)
:=
\left(
1+
\frac{|\pi_v x-\rho m_1e_1-\rho m_2e_2|}{\rho}
+
|x\cdot v-j|
\right)^{-N},
\]
where $\pi_v$ is the orthogonal projection onto $v^\perp$.

If $N>4$, then there exists a constant $C_N\ge1$, independent of $\rho$ and of $v$, such that
\[
\sup_{x\in\mathbb R^3}
\sum_{m\in\mathbb Z^2}
\sum_{j\in\mathbb Z}
w_{m,j}(x)
\le C_N.
\]
\end{lem}

\begin{proof}
Fix $x\in\mathbb R^3$ and write
\[
x=y+sv,
\qquad
y=\pi_vx\in v^\perp,\quad s=x\cdot v.
\]
Then
\[
\sum_{m\in\mathbb Z^2}
\sum_{j\in\mathbb Z}
w_{m,j}(x)
=
\sum_{m\in\mathbb Z^2}
\sum_{j\in\mathbb Z}
\left(
1+
\left|\rho^{-1}y-m_1e_1-m_2e_2\right|
+
|s-j|
\right)^{-N}.
\]
Using
\[
1+A+B\ge (1+A)^{1/2}(1+B)^{1/2},
\]
we obtain
\[
(1+A+B)^{-N}
\le
(1+A)^{-N/2}(1+B)^{-N/2}.
\]
Therefore,
\[
\sum_{m,j}w_{m,j}(x)
\le
\left[
\sum_{m\in\mathbb Z^2}
\left(
1+\left|\rho^{-1}y-m_1e_1-m_2e_2\right|
\right)^{-N/2}
\right]
\left[
\sum_{j\in\mathbb Z}
(1+|s-j|)^{-N/2}
\right].
\]
The first sum is uniformly bounded if $N/2>2$, and the second if $N/2>1$. Thus it is enough to take $N>4$, as assumed.
\end{proof}

\begin{cor}[uniform verification of extreme Schur in parallel families]\label{cor:quadratic-extreme-schur-parallel}
Let $S_Q$, $\mu_Q$, and $\rho$ be as in Section~\ref{sec:bad-geometry}. Fix a compact chart $U_0\Subset U$ as in Subsection~\ref{subsec:elliptic-uniform-extreme-geometry}. Let $v\in S^2$ be a relevant direction for the patch, and let $\mathbb T_v$ be a family of tubes of length comparable to $1$, transverse radius comparable to $\rho$, and direction parallel to $v$, parametrized by a transverse lattice of spacing comparable to $\rho$ in $v^\perp$ and by longitudinal windows of length comparable to $1$.

Suppose that the family is indexed as
\[
\mathbb T_v=\{T_{m,j}:m\in\mathbb Z^2,\ j\in\mathbb Z\},
\]
with respect to an orthonormal basis $(e_1,e_2,v)$, and that, for each $T_{m,j}\in\mathbb T_v$, the function $\phi_{m,j}$ satisfies
\[
|\phi_{m,j}(x)|
\le
A\rho^{-1}
\left(
1+
\frac{|\pi_v x-\rho m_1e_1-\rho m_2e_2|}{\rho}
+
|x\cdot v-j|
\right)^{-N}
\qquad
\text{for every }x\in\mathbb R^3,
\]
with $N>4$ and $A\ge1$. Then there exists a constant $C\ge1$, depending only on $N$, on $A$, on $Q$, on $U_0$, on $\|\chi\|_{L^\infty}$, and on the geometric constants of the tubular family, but independent of $\rho$, of $v$, and of the family $\mathbb T_v$, such that
\[
\sup_{T\in\mathbb T_v}
\sum_{T'\in\mathbb T_v}
\left|
\int_{E_v^{\mathrm{ext}}}
\phi_T(x)\overline{\phi_{T'}(x)}
\,\mathrm d\mu_Q(x)
\right|
\le
C\rho^{-1/2},
\]
and also
\[
\sup_{T'\in\mathbb T_v}
\sum_{T\in\mathbb T_v}
\left|
\int_{E_v^{\mathrm{ext}}}
\phi_T(x)\overline{\phi_{T'}(x)}
\,\mathrm d\mu_Q(x)
\right|
\le
C\rho^{-1/2}.
\]
Consequently, the Schur hypothesis of Proposition~\ref{prop:quadratic-extreme-family-schur} is verified for $\mathbb T_v$ with a constant $C_{\mathrm{Sch}}\le C$ uniform in $v$ and in $R$.
\end{cor}

\begin{proof}
For $T,T'\in\mathbb T_v$, define
\[
K(T,T')
:=
\int_{E_v^{\mathrm{ext}}}
\phi_T(x)\overline{\phi_{T'}(x)}
\,\mathrm d\mu_Q(x).
\]
By spatial adaptation,
\[
|K(T,T')|
\le
A^2\rho^{-2}
\int_{E_v^{\mathrm{ext}}}
w_T(x)w_{T'}(x)\,\mathrm d\mu_Q(x),
\]
where
\[
w_{m,j}(x)
:=
\left(
1+
\frac{|\pi_v x-\rho m_1e_1-\rho m_2e_2|}{\rho}
+
|x\cdot v-j|
\right)^{-N}.
\]
Summing in $T'$ and using Lemma~\ref{lem:parallel-tube-tail-overlap},
\[
\sum_{T'\in\mathbb T_v}|K(T,T')|
\le
A^2\rho^{-2}
\int_{E_v^{\mathrm{ext}}}
w_T(x)
\sum_{m',j'}w_{m',j'}(x)
\,\mathrm d\mu_Q(x)
\]
\[
\le
C A^2\rho^{-2}
\int_{E_v^{\mathrm{ext}}}
w_T(x)\,\mathrm d\mu_Q(x).
\]

We now estimate the last integral uniformly in the relevant direction $v$. Let $T=T_{m,j}$.
For dyadic integers $k\ge0$, set
\[
A_k(T)
:=
\left\{
x\in E_v^{\mathrm{ext}}:
2^{k-1}\rho\le \operatorname{dist}(x,T)<2^k\rho
\right\},
\]
with the evident modification for $k=0$. If $2^k\rho\le\sigma_0$, Lemma~\ref{lem:elliptic-tubular-sublevel}, applied with $\sigma=2^k\rho$, gives
\[
\mu_Q(A_k(T))
\le
\mu_Q\bigl(E_v^{\mathrm{ext}}\cap N_{2^k\rho}(T)\bigr)
\le
C\rho^{1/2}(2^k\rho)
=
C2^k\rho^{3/2}.
\]
Therefore, in this range,
\[
\int_{A_k(T)} w_T(x)\,\mathrm d\mu_Q(x)
\le
C2^{-Nk}2^k\rho^{3/2}
=
C2^{(1-N)k}\rho^{3/2}.
\]
Since $N>4$, the sum over all $k$ with $2^k\rho\le\sigma_0$ is bounded by $C\rho^{3/2}$.

For the remaining range $2^k\rho>\sigma_0$, the compact support of $\mu_Q$ and the rapid decay of $w_T$ give
\[
\sum_{2^k\rho>\sigma_0}
\int_{A_k(T)} w_T(x)\,\mathrm d\mu_Q(x)
\le
C\sum_{2^k\rho>\sigma_0}2^{-Nk}
\le
C_N \rho^N
\le
C_N\rho^{3/2},
\]
after increasing $C_N$, since $0<\rho\le1$ and $N>4$. Hence
\[
\int_{E_v^{\mathrm{ext}}}
w_T(x)\,\mathrm d\mu_Q(x)
\le
C\rho^{3/2}.
\]

Thus,
\[
\sum_{T'\in\mathbb T_v}|K(T,T')|
\le
C\rho^{-2}\rho^{3/2}
=
C\rho^{-1/2}.
\]
The symmetric estimate is obtained by interchanging the roles of $T$ and $T'$, using again the same pointwise overlap and the same uniform tail mass bound.
\end{proof}

\begin{remark}[boxwise diagonal reading of extreme Schur]\label{rem:quadratic-extreme-schur-boxwise}
Corollary~\ref{cor:quadratic-extreme-schur-parallel} verifies the Schur hypothesis of Proposition~\ref{prop:quadratic-extreme-family-schur} for a parallel tubular family associated with an arbitrary relevant direction of the patch, with constants uniform in that direction. This verification is applied box by box: for a frequency box $\Theta$, the wave packet family $\mathbb T(\Theta)$ has essentially fixed direction $v_\Theta$, and the Schur control applies to that family.

The extreme estimate used in Theorem~\ref{thm:quadratic-diagonal-sharp} is diagonal in $\Theta$:
\[
\int
\sum_{\Theta\in\mathbb E}|f_\Theta(x)|^2\,\mathrm d\mu_Q(x).
\]
Therefore, no cross terms between distinct frequency boxes appear. The extreme argument required for the assembly consists of applying the family bound inside each box $\Theta\in\mathbb E$ and then summing in $\Theta$ using the corresponding Bessel inequality.
\end{remark}

\begin{prop}[boxwise diagonal extreme bound]\label{prop:quadratic-extreme-boxwise-diagonal}
Let $\Theta\in\mathbb E_{\tau_0}$ be a nontransversal box and let $v_\Theta$ be its associated direction. Consider the wave packet decomposition of $f_\Theta$ given by Proposition~\ref{prop:standard-wave-packet-decomposition},
\[
f_\Theta
=
\sum_{T\in\mathbb T(\Theta)} c_T\phi_T.
\]
Define
\[
E_{v_\Theta}^{\mathrm{ext}}
:=
\{x\in S_Q: |n_Q(x)\cdot v_\Theta|\le \rho^{1/2}\}.
\]
Then
\[
\int_{E_{v_\Theta}^{\mathrm{ext}}}
|f_\Theta(x)|^2\,\mathrm d\mu_Q(x)
\le
C\rho^{-1/2}
\|f_\Theta\|_{L^2(\mathbb R^3)}^2,
\]
where $C$ depends only on $Q$, on the chart, on $\|\chi\|_{L^\infty}$, on the geometric constants of the tubular family, and on the structural constants of the wave packet decomposition in Proposition~\ref{prop:standard-wave-packet-decomposition}.
\end{prop}

\begin{proof}
If $E_{v_\Theta}^{\mathrm{ext}}=\varnothing$, there is nothing to prove. Otherwise, the direction $v_\Theta$ enters the extreme regime of the patch, and the uniform elliptic form of Subsection~\ref{subsec:elliptic-uniform-extreme-geometry} applies.

By Proposition~\ref{prop:standard-wave-packet-decomposition}, the directions of the tubes in $\mathbb T(\Theta)$ are contained in an angular cone of aperture $O(\rho)$ around $v_\Theta$. Using Lemma~\ref{lem:central-direction-tube-stability}, each tube can be replaced, at the cost of dilating the geometric constants and the tails by a uniform factor, by a tube parallel to $v_\Theta$. The spatial discretization of Proposition~\ref{prop:standard-wave-packet-decomposition} produces a complete net of indices $(m,j)\in\mathbb Z^2\times\mathbb Z$: the index $m$ parametrizes the transverse lattice of spacing comparable to $\rho$ in $v_\Theta^\perp$, and the index $j$ parametrizes the longitudinal windows of length comparable to $1$. In integrals over $\operatorname{supp}\mu_Q$, the longitudinal sum has uniformly bounded effective multiplicity, whereas the transverse sum has the two-dimensional growth of Lemma~\ref{lem:parallel-tube-tail-overlap}. Therefore, the resulting family satisfies the geometric hypothesis required for the Schur verification.

By Corollary~\ref{cor:quadratic-extreme-schur-parallel} and the uniform elliptic form of Subsection~\ref{subsec:elliptic-uniform-extreme-geometry}, the tubular family $\mathbb T(\Theta)$ satisfies the Schur hypotheses of Proposition~\ref{prop:quadratic-extreme-family-schur}, with constant uniform in $\Theta$ and in $R$.

Applying Proposition~\ref{prop:quadratic-extreme-family-schur} to the tubular family associated with $\Theta$, one obtains
\[
\int_{E_{v_\Theta}^{\mathrm{ext}}}
|f_\Theta(x)|^2\,\mathrm d\mu_Q(x)
\le
C\rho^{-1/2}
\sum_{T\in\mathbb T(\Theta)}|c_T|^2.
\]
The Bessel inequality of Proposition~\ref{prop:standard-wave-packet-decomposition} gives
\[
\sum_{T\in\mathbb T(\Theta)}|c_T|^2
\lesssim
\|f_\Theta\|_{L^2(\mathbb R^3)}^2.
\]
Combining the two estimates gives the claim.
\end{proof}

\begin{lem}[quadratic projective multiplicity outside the extreme strip]\label{lem:quadratic-nonextreme-projection-multiplicity}
Let
\[
S_Q
=
\{(u_1,u_2,Q(u_1,u_2)):u\in U\},
\qquad
Q(u_1,u_2)
=
\frac12(\lambda_1u_1^2+\lambda_2u_2^2),
\qquad
\lambda_1\lambda_2>0.
\]
Let $U_0\Subset U$ be a compact chart and set
\[
S_{Q,0}:=X(U_0),
\qquad
X(u_1,u_2)=(u_1,u_2,Q(u_1,u_2)).
\]
For $v\in S^2$, let
\[
\pi_v:\mathbb R^3\to v^\perp
\]
be the orthogonal projection and define
\[
E_v^{\mathrm{ext}}
:=
\{x\in S_{Q,0}: |n_Q(x)\cdot v|\le \rho^{1/2}\}.
\]
Then, for every $z\in v^\perp$,
\[
\#
\left(
(S_{Q,0}\setminus E_v^{\mathrm{ext}})
\cap
\pi_v^{-1}(\{z\})
\right)
\le 2.
\]
The bound is uniform in $v$, in $z$, in $\rho$, and in the compact chart $U_0$.
\end{lem}

\begin{proof}
Fix $v\in S^2$ and $z\in v^\perp$. The fiber of $\pi_v$ over $z$ is the line
\[
\ell_{z,v}:=\{z+tv:t\in\mathbb R\}.
\]
The points of $S_Q\cap \ell_{z,v}$ correspond to solutions of
\[
z+tv=(u_1,u_2,Q(u_1,u_2)).
\]
Writing $z=(z_1,z_2,z_3)$ and $v=(v_1,v_2,v_3)$, the first two coordinates impose
\[
u_1=z_1+tv_1,
\qquad
u_2=z_2+tv_2.
\]
The condition of belonging to $S_Q$ is then reduced to
\[
z_3+tv_3
=
Q(z_1+tv_1,z_2+tv_2),
\]
that is,
\[
Q(z_1+tv_1,z_2+tv_2)-z_3-tv_3=0.
\]
The left-hand side is a polynomial of degree at most $2$ in $t$.

We rule out that this polynomial is identically zero. Indeed, its quadratic coefficient is
\[
\frac12(\lambda_1v_1^2+\lambda_2v_2^2).
\]
Since $\lambda_1\lambda_2>0$, the quadratic form $\lambda_1a_1^2+\lambda_2a_2^2$ is definite. If the quadratic coefficient vanishes, then $v_1=v_2=0$. Since $v\in S^2$, one has $v_3=\pm1$, and the linear coefficient of the polynomial contains the term $-v_3t$, so it cannot vanish identically either.

Thus the line $\ell_{z,v}$ intersects $S_Q$ in at most two points. Restricting to $S_{Q,0}\setminus E_v^{\mathrm{ext}}$ can only decrease the cardinality. This proves the claim.
\end{proof}

\begin{prop}[degenerate projective bound outside the extreme strip]\label{prop:quadratic-nonextreme-projective-test}
Let
\[
S_Q
=
\{(u_1,u_2,Q(u_1,u_2)):u\in U\},
\qquad
Q(u_1,u_2)
=
\frac12(\lambda_1u_1^2+\lambda_2u_2^2),
\qquad
\lambda_1\lambda_2>0.
\]
Let
\[
\mu_Q=\chi\,\mathcal H^2\lfloor S_Q,
\]
where $\chi\in L^\infty(S_Q)$ is compactly supported in $S_{Q,0}=X(U_0)$, with $U_0\Subset U$. For $v\in S^2$, define
\[
E_v^{\mathrm{ext}}
:=
\{x\in S_{Q,0}: |n_Q(x)\cdot v|\le \rho^{1/2}\}.
\]
Then there exists a constant $C\ge1$, depending only on $Q$, on $U_0$, and on $\|\chi\|_{L^\infty}$, such that for every $v\in S^2$, every $z\in v^\perp$, and every $0<r\le1$,
\[
(\pi_v)_\#
\bigl(
\mu_Q\lfloor(S_{Q,0}\setminus E_v^{\mathrm{ext}})
\bigr)
\bigl(B_{v^\perp}(z,r)\bigr)
\le
C\rho^{-1/2}r^2.
\]
\end{prop}

\begin{proof}
Let $A:=B_{v^\perp}(z,r)$. Since $\chi$ is supported in $S_{Q,0}$,
\[
(\pi_v)_\#
\bigl(
\mu_Q\lfloor(S_{Q,0}\setminus E_v^{\mathrm{ext}})
\bigr)(A)
\le
\|\chi\|_{L^\infty}
\mathcal H^2
\bigl((S_{Q,0}\setminus E_v^{\mathrm{ext}})\cap \pi_v^{-1}(A)\bigr).
\]

The tangential Jacobian of the orthogonal projection
\[
\pi_v|_{S_Q}:S_Q\to v^\perp
\]
satisfies
\[
J_{S_Q}\pi_v(x)=|n_Q(x)\cdot v|.
\]
On $S_{Q,0}\setminus E_v^{\mathrm{ext}}$ one has
\[
|n_Q(x)\cdot v|>\rho^{1/2}.
\]
Therefore,
\[
1\le \rho^{-1/2}J_{S_Q}\pi_v(x)
\qquad
\text{on }S_{Q,0}\setminus E_v^{\mathrm{ext}}.
\]
Hence
\[
\mathcal H^2
\bigl((S_{Q,0}\setminus E_v^{\mathrm{ext}})\cap \pi_v^{-1}(A)\bigr)
\le
\rho^{-1/2}
\int_{(S_{Q,0}\setminus E_v^{\mathrm{ext}})\cap \pi_v^{-1}(A)}
J_{S_Q}\pi_v(x)\,\mathrm d\mathcal H^2(x).
\]

Applying the area formula in the notation fixed before Lemma~\ref{lem:projective-area-formula-bound} to
$\pi_v|_{S_Q}$,
\[
\int_{(S_{Q,0}\setminus E_v^{\mathrm{ext}})\cap \pi_v^{-1}(A)}
J_{S_Q}\pi_v(x)\,\mathrm d\mathcal H^2(x)
=
\int_A
N(y,\pi_v,S_{Q,0}\setminus E_v^{\mathrm{ext}})
\,\mathrm d\mathcal H^2_{v^\perp}(y).
\]
By Lemma~\ref{lem:quadratic-nonextreme-projection-multiplicity},
\[
N(y,\pi_v,S_{Q,0}\setminus E_v^{\mathrm{ext}})
\le 2
\]
for every $y\in v^\perp$. Thus
\[
\int_A
N(y,\pi_v,S_{Q,0}\setminus E_v^{\mathrm{ext}})
\,\mathrm d\mathcal H^2_{v^\perp}(y)
\le
2\,\mathcal H^2_{v^\perp}(A)
\lesssim r^2.
\]
Combining the preceding inequalities gives
\[
(\pi_v)_\#
\bigl(
\mu_Q\lfloor(S_{Q,0}\setminus E_v^{\mathrm{ext}})
\bigr)
\bigl(B_{v^\perp}(z,r)\bigr)
\le
C\rho^{-1/2}r^2.
\]
\end{proof}

\begin{remark}[intermediate regime and degenerate projective cost]\label{rem:quadratic-intermediate-projected-cost}
The complement of the extreme strip
\[
S_Q\setminus E_{v_\Theta}^{\mathrm{ext}}
=
\{x\in S_Q: |n_Q(x)\cdot v_\Theta|>\rho^{1/2}\}
\]
contains, in particular, the intermediate regime
\[
\rho^{1/2}<|n_Q(x)\cdot v_\Theta|<\tau_0.
\]
This region is not treated by a uniform transversal bound. The argument uses only the degenerate lower bound
\[
|n_Q(x)\cdot v_\Theta|>\rho^{1/2},
\]
which produces, through the area formula for the projection $\pi_{v_\Theta}$ and the uniform projective multiplicity of the elliptic model, the projective cost
\[
C_{\mathrm{proj}}\lesssim \rho^{-1/2}.
\]
Thus the intermediate regime is included in the non-extreme block with the same admissible quadratic cost $\rho^{-1/2}$ as the extreme strip. No uniform improvement is asserted in that regime.
\end{remark}

\begin{prop}[projective bound outside the extreme strip]\label{prop:quadratic-nonextreme-projected-cost}
Let $S_Q$, $\mu_Q$, and $\rho$ be as in Section~\ref{sec:bad-geometry}. Fix a frequency box $\Theta$ and let $v_\Theta\in S^2$ be its associated physical direction. Define
\[
E_{v_\Theta}^{\mathrm{ext}}
:=
\{x\in S_Q: |n_Q(x)\cdot v_\Theta|\le \rho^{1/2}\}.
\]
Define the restricted measure
\[
\mu_{Q,\Theta}^{\mathrm{nex}}
:=
\mu_Q\lfloor (S_Q\setminus E_{v_\Theta}^{\mathrm{ext}}).
\]
Consider the wave packet decomposition of $f_\Theta$ given by Proposition~\ref{prop:standard-wave-packet-decomposition}. Then there exists a constant $C\ge 1$, depending only on $Q$, on the chart, on $\|\chi\|_{L^\infty}$, and on the structural constants of the wave packet decomposition, such that
\[
\int_{S_Q\setminus E_{v_\Theta}^{\mathrm{ext}}}
|f_\Theta(x)|^2\,\mathrm d\mu_Q(x)
\le
C\rho^{-1/2}
\|f_\Theta\|_{L^2(\mathbb R^3)}^2.
\]
\end{prop}

\begin{proof}
By Proposition~\ref{prop:quadratic-nonextreme-projective-test}, applied with $v=v_\Theta$, one has
\[
(\pi_{v_\Theta})_\#\mu_{Q,\Theta}^{\mathrm{nex}}
\bigl(B_{v_\Theta^\perp}(z,r)\bigr)
\le
C\rho^{-1/2}r^2
\]
for every $z\in v_\Theta^\perp$ and every $0<r\le1$. That is, the restricted measure $\mu_{Q,\Theta}^{\mathrm{nex}}$ satisfies the projective hypothesis $(P_2)$ with constant
\[
C_{\mathrm{proj}}\lesssim \rho^{-1/2}.
\]

The tubular family $\mathbb T(\Theta)$ associated with the standard decomposition satisfies the dyadic overlap hypothesis $(O_2)$, by Lemma~\ref{lem:dyadic-overlap-standard}, with constants independent of $\Theta$ and $R$. Moreover, since $\chi$ is compactly supported in the fixed chart, the measure $\mu_{Q,\Theta}^{\mathrm{nex}}$ is supported in a fixed compact set; after a fixed normalization of the chart, it lies in the support regime required by Proposition~\ref{prop:diagonal-under-projection}.

Applying Proposition~\ref{prop:diagonal-under-projection} to the measure $\mu_{Q,\Theta}^{\mathrm{nex}}$, with $\beta=2$, with projective constant $C_{\mathrm{proj}}\lesssim\rho^{-1/2}$, and with the standard decomposition of Proposition~\ref{prop:standard-wave-packet-decomposition}, gives
\[
\int_{S_Q\setminus E_{v_\Theta}^{\mathrm{ext}}}
|f_\Theta(x)|^2\,\mathrm d\mu_Q(x)
\le
C\rho^{-1/2}
\sum_{T\in\mathbb T(\Theta)}|c_T|^2.
\]
The Bessel inequality of Proposition~\ref{prop:standard-wave-packet-decomposition} implies
\[
\sum_{T\in\mathbb T(\Theta)}|c_T|^2
\lesssim
\|f_\Theta\|_{L^2(\mathbb R^3)}^2.
\]
This concludes the proof.
\end{proof}

\begin{prop}[total bound per nontransversal box]\label{prop:quadratic-nontransversal-box-cost}
Let $\Theta\in\mathbb E_{\tau_0}$ and let $v_\Theta$ be its associated direction. Consider the wave packet decomposition of $f_\Theta$ given by Proposition~\ref{prop:standard-wave-packet-decomposition},
\[
f_\Theta
=
\sum_{T\in\mathbb T(\Theta)} c_T\phi_T.
\]
Define
\[
E_{v_\Theta}^{\mathrm{ext}}
:=
\{x\in S_Q: |n_Q(x)\cdot v_\Theta|\le \rho^{1/2}\}.
\]
Then
\[
\int_{S_Q}
|f_\Theta(x)|^2\,\mathrm d\mu_Q(x)
\le
C\rho^{-1/2}
\|f_\Theta\|_{L^2(\mathbb R^3)}^2,
\]
where $C$ depends only on $Q$, on the chart, on $\|\chi\|_{L^\infty}$, on the geometric constants of the tubular family, and on the structural constants of the wave packet decomposition in Proposition~\ref{prop:standard-wave-packet-decomposition}.
\end{prop}

\begin{proof}
Decompose
\[
S_Q
=
E_{v_\Theta}^{\mathrm{ext}}
\cup
(S_Q\setminus E_{v_\Theta}^{\mathrm{ext}}).
\]
By Proposition~\ref{prop:quadratic-extreme-boxwise-diagonal},
\[
\int_{E_{v_\Theta}^{\mathrm{ext}}}
|f_\Theta(x)|^2\,\mathrm d\mu_Q(x)
\le
C\rho^{-1/2}
\|f_\Theta\|_{L^2(\mathbb R^3)}^2.
\]
By Proposition~\ref{prop:quadratic-nonextreme-projected-cost},
\[
\int_{S_Q\setminus E_{v_\Theta}^{\mathrm{ext}}}
|f_\Theta(x)|^2\,\mathrm d\mu_Q(x)
\le
C\rho^{-1/2}
\|f_\Theta\|_{L^2(\mathbb R^3)}^2.
\]
Summing the two estimates gives the claim.
\end{proof}

\begin{thm}[diagonal bound with optimal loss]\label{thm:quadratic-diagonal-sharp}
Let
\[
\mu_Q=\chi\,\mathcal H^2\lfloor S_Q
\]
be as above, with $\lambda_1\lambda_2>0$, and let $\rho=R^{-1/2}$. Fix $\tau_0\in(0,1)$ and consider the transversal--nontransversal partition
\[
\{\Theta\}
=
\mathbb T_{\tau_0}\,\dot\cup\,\mathbb E_{\tau_0}
\]
from Definition~\ref{def:quadratic-transversal-extreme-partition}. Then there exists a constant $C\ge1$, depending on $\tau_0$, on $Q$, on the chart, on $\|\chi\|_{L^\infty}$, and on the structural constants of the wave packet decomposition, but independent of $R$, such that
\[
\int_{\mathbb R^3}
\sum_\Theta |f_\Theta(x)|^2\,\mathrm d\mu_Q(x)
\le
C\rho^{-1/2}
\sum_\Theta
\|f_\Theta\|_{L^2(\mathbb R^3)}^2.
\]
Consequently,
\[
\|\mathcal G_R f\|_{L^2(\mathrm d\mu_Q)}
\le
C\rho^{-1/4}
\left(
\sum_\Theta
\|f_\Theta\|_{L^2(\mathbb R^3)}^2
\right)^{1/2}.
\]
By almost orthogonality of the frequency partition,
\[
\|\mathcal G_R f\|_{L^2(\mathrm d\mu_Q)}
\le
C\rho^{-1/4}
\|f\|_{L^2(\mathbb R^3)}
=
C R^{1/8}
\|f\|_{L^2(\mathbb R^3)}.
\]
\end{thm}

\begin{proof}
First consider the transversal subfamily $\mathbb T_{\tau_0}$. By the definition of $\mathbb T_{\tau_0}$,
\[
|n_Q(x)\cdot v_\Theta|\ge \tau_0
\qquad
\text{for every }x\in \operatorname{supp}\chi
\]
and every $\Theta\in\mathbb T_{\tau_0}$. By Lemma~\ref{lem:surface-projection-transverse}, applied in the compact quadratic chart and using the lower bound $|n_Q(x)\cdot v_\Theta|\ge\tau_0$, the projected measure $(\pi_{v_\Theta})_\#\mu_Q$ satisfies the hypothesis $(P_2)$ with constant depending on $\tau_0$, on $Q$, on the chart, and on $\|\chi\|_{L^\infty}$, but independent of $R$ and of $\Theta\in\mathbb T_{\tau_0}$. Applying Proposition~\ref{prop:diagonal-under-projection} with $\beta=2$ to the transversal subfamily gives
\[
\int_{\mathbb R^3}
\sum_{\Theta\in\mathbb T_{\tau_0}}
|f_\Theta(x)|^2\,\mathrm d\mu_Q(x)
\le
C_{\tau_0}
\sum_{\Theta\in\mathbb T_{\tau_0}}
\|f_\Theta\|_{L^2(\mathbb R^3)}^2.
\]
Since $\rho\le1$, this contribution is also bounded by
\[
C_{\tau_0}\rho^{-1/2}
\sum_{\Theta\in\mathbb T_{\tau_0}}
\|f_\Theta\|_{L^2(\mathbb R^3)}^2.
\]

For the nontransversal subfamily $\mathbb E_{\tau_0}$, apply Proposition~\ref{prop:quadratic-nontransversal-box-cost} to each box $\Theta\in\mathbb E_{\tau_0}$. Summing in $\Theta$ gives
\[
\int_{\mathbb R^3}
\sum_{\Theta\in\mathbb E_{\tau_0}}
|f_\Theta(x)|^2\,\mathrm d\mu_Q(x)
\le
C\rho^{-1/2}
\sum_{\Theta\in\mathbb E_{\tau_0}}
\|f_\Theta\|_{L^2(\mathbb R^3)}^2.
\]

Assembling the two contributions by Proposition~\ref{prop:quadratic-two-piece-assembly} gives
\[
\int_{\mathbb R^3}
\sum_\Theta |f_\Theta(x)|^2\,\mathrm d\mu_Q(x)
\le
C\rho^{-1/2}
\sum_\Theta
\|f_\Theta\|_{L^2(\mathbb R^3)}^2.
\]
Taking square roots,
\[
\|\mathcal G_R f\|_{L^2(\mathrm d\mu_Q)}
\le
C\rho^{-1/4}
\left(
\sum_\Theta
\|f_\Theta\|_{L^2(\mathbb R^3)}^2
\right)^{1/2}.
\]
The $L^2$ almost orthogonality of the frequency partition implies
\[
\left(
\sum_\Theta
\|f_\Theta\|_{L^2(\mathbb R^3)}^2
\right)^{1/2}
\lesssim
\|f\|_{L^2(\mathbb R^3)}.
\]
Since $\rho=R^{-1/2}$, one has $\rho^{-1/4}=R^{1/8}$.
\end{proof}

\begin{remark}[scope of the upper bound]\label{rem:quadratic-upper-bound-status}
Theorem~\ref{thm:quadratic-diagonal-sharp} assembles the uniform transversal bound for $\mathbb T_{\tau_0}$ with the boxwise nontransversal bound of Proposition~\ref{prop:quadratic-nontransversal-box-cost} for $\mathbb E_{\tau_0}$.

For each box, the nontransversal part is divided into the extreme strip $E_{v_\Theta}^{\mathrm{ext}}$ and its complement. The extreme strip is controlled by tubular Schur and the extreme tail mass; the complement is controlled by degenerate projection with threshold $\rho^{1/2}$. Both mechanisms have admissible quadratic cost $\rho^{-1/2}$.

The resulting norm loss is $\rho^{-1/4}=R^{1/8}$. Remark~\ref{rem:quadratic-sharpness} states that this loss cannot be removed in the quadratic model with positive surface density under the fixed microlocal normalization.
\end{remark}

\begin{remark}[optimality of the exponent]\label{rem:quadratic-sharpness}
Proposition~\ref{prop:quadratic-extreme-wave-packet-cost} and Corollary~\ref{cor:quadratic-no-uniform-bound} show that the loss
\[
\rho^{-1/4}=R^{1/8}
\]
cannot be removed in the quadratic model with positive surface density, provided the class of functions contains the tangent wave packet specified there. Therefore, the upper bound in Theorem~\ref{thm:quadratic-diagonal-sharp} is optimal at the norm scale under the fixed microlocal normalization.
\end{remark}

\bibliographystyle{plain}
\bibliography{SqFF}

\end{document}